\documentclass[preprint,12pt]{elsarticle}

\usepackage{graphicx} 					% Required for inserting images
\usepackage{subcaption}					% Subfigure captions
\usepackage{amsmath} 					% Math symbols
\usepackage{amssymb} 					% Math symbols
\usepackage{bm} 						% Bold vectors
\usepackage{spalign} 					% Matlab-style matrices and vectors
\usepackage{mathtools} 					% Delimiters for \Set command
\usepackage{algorithm} 					% Pseudocode
\usepackage{algpseudocode} 				% Pseudocode
\usepackage{booktabs}					% Table rules
\usepackage{enumitem}       			% Enumerate labels
\usepackage{numprint}					% Thousand separators
\usepackage{amsthm}					    % Theorem styles
\usepackage{svg}
\usepackage{url}
\usepackage{cleveref}

\setenumerate[1]{label={(\arabic*)}} 	% Enumerate with brackets
\npthousandsep{\,}						% separate thousands with space

\newcommand{\bs}[1]{\boldsymbol{#1}}
\newcommand{\Th}[0]{\mathcal{T}}
\newcommand{\Sh}[0]{\mathcal{S}}

\newcommand{\Zh}[0]{\mathcal{Z}}

\newcommand{\R}[0]{\mathbb{R}}
\newcommand{\N}[0]{\mathbb{N}}

\let\mat=\spalignmat

\let\vec=\spalignvector

\newcommand{\red}[1]{\textcolor{black}{#1}}

\newcommand{\vertiii}[1]{{\left\vert\kern-0.25ex\left\vert\kern-0.25ex\left\vert #1 \right\vert\kern-0.25ex\right\vert\kern-0.25ex\right\vert}}
\newcommand{\norm}[2][]{\| #2 \|_{#1}}

\DeclarePairedDelimiterX\Set[1]{\lbrace}{\rbrace}%
 {  #1 }

\newcommand{\expnumber}[2]{{#1}\mathrm{e}{#2}}

\newcommand{\renewtheorem}[1]{%
  \expandafter\let\csname #1\endcsname\relax
  \expandafter\let\csname c@#1\endcsname\relax
  \expandafter\let\csname end#1\endcsname\relax
  \newtheorem{#1}%
}

\theoremstyle{plain}
\renewtheorem{thrm}{Theorem}[section]
\renewtheorem{lmm}[thrm]{Lemma}
\renewtheorem{crllr}[thrm]{Corollary}
\renewtheorem{prpstn}[thrm]{Proposition}
\renewtheorem{crtrn}[thrm]{Criterion}
\theoremstyle{definition}
\renewtheorem{dfntn}[thrm]{Definition}
\renewtheorem{prblm}[thrm]{Problem}

\newenvironment{theorem}{%
  \begin{thrm}%
}{%
  \end{thrm}%
}

\newenvironment{problem}{%
  \begin{prblm}%
}{%
  \end{prblm}%
}

\newenvironment{definition}{%
  \begin{dfntn}%
}{%
  \end{dfntn}%
}

\newenvironment{lemma}{%
  \begin{lmm}%
}{%
  \end{lmm}%
}

\newenvironment{assumption}{%
  \begin{crtrn}%
}{%
  \end{crtrn}%
}

\newtheorem{remark}{Remark}

\numberwithin{equation}{section}

\journal{Computers \& Mathematics with Applications}

\begin{document}

% \maketitle
\begin{frontmatter}

\title{Distributed finite element solution using model order reduction\tnotemark[1]}

\tnotetext[1]{This work was supported by the Academy of Finland (Decision 353080)}

\author[label1]{Tom Gustafsson} %% Author name
			\ead{tom.gustafsson@aalto.fi}
\affiliation[label1]{organization={Department of Mechanical Engineering, Aalto University},%Department and Organization
            addressline={P.O. Box 11100 00076 Aalto}, 
            city={Espoo},
            country={Finland},
			}

\author[label2]{Antti Hannukainen} %% Author name
			\ead{antti.hannukainen@aalto.fi}
\affiliation[label2]{organization={Department of Mathematics and Systems Analysis, Aalto University}%Department and Organization
			% email={antti.hannukainen@aalto.fi},
            }

\author[label2]{Vili Kohonen} %% Author name
			\ead{vili.kohonen@aalto.fi}

\begin{abstract}
We extend a localized model order reduction method
for the distributed finite element solution of elliptic boundary value problems in the cloud.
% using model order reduction on each subdomain
% of the domain decomposition.
% We give a constructive proof how to compute the local optimal bases via a weighted low-rank matrix approximation with finite element matrices efficiently, making the finite element implementation compatible with the continuous problem. 
We give a computationally efficient technique to compute the required inner product matrices and optimal reduced bases.
A memory-efficient methodology is proposed to project the global finite element linear system onto the reduced basis. 
% These developments enable the usage of the method for large-scale problems with cloud computing.
% \red{
 % We show the equivalence of an operator norm 
 % and a weighted $\ell^2$ matrix norm for a 
 % discrete lifting operator, and how the 
 % global reduction error in $H^1$ norm 
 % can be controlled 
 % via a weighted low-rank matrix approximation.
% }
% We show that the error of the reduction
% can be succesfully controlled
% in the $H^1$ norm due to the properly
% weighted $l^2$ low-rank approximation.
Our
numerical results demonstrate the technique
using \red{nontrivial} tetrahedral meshes \red{and subdomain interfaces}
with up to 85 million degrees-of-freedom
on a laptop computer
by distributing 
the bulk of the model order reduction to the cloud.
\end{abstract}

% \begin{highlights}
% \item A distributed finite element method is extended for cloud computing
% \item Local optimal basis construction is formulated as a weighted low-rank matrix approximation
% \item Three poorly scaling substeps of the applied model order technique are identified and improved approaches are proposed
% \item Numerical tests up to 85M degrees-of-freedom on a laptop and the cloud validate the theoretical estimates
% \end{highlights}

% \begin{keywords}

\end{frontmatter}

\section{Introduction}

High performance computing (HPC)
is often done using supercomputers
that have fast internal networking.
There is a focus on networking because the bottleneck
in many HPC workloads
can be traced back to message passing and sharing
of data between the processes.
Consequently, many distributed workloads
benefit from
specialized networking technologies or topologies
that improve the node-to-node bandwidth
and latency; see, e.g., \cite{zhou_hardware_2007}.
Cloud computing, on the other hand,
is a
paradigm for resource allocation
in data centers
which relies on the virtualization technology.
Commercial cloud capacity is ubiquitous,
and the use of excess capacity can be cost-effective
as many providers support a spot pricing model~\cite{munhoz_faulttolerant_2022}.
However,
cloud services have different priorities than supercomputers
when it comes to energy efficiency, security and networking.
As a consequence,
cloud computing
may be more suitable for tasks
with lower requirements for interprocess communication~\cite{coghlan_magellan_2011}.

This study focuses on a distributed finite element solution
technique for elliptic problems, proposed in \cite{babuska_optimal_2011}, which
eliminates the need for
message passing during the solution
of the linear system.
The domain is decomposed into several overlapping subdomains as a preprocessing step
and a local reduced basis is computed
for each subdomain in parallel using randomized model order reduction
techniques~\cite{buhr_randomized_2018}.
%The dimensionality of the reduced basis is proportional to the number of 
%boundary nodes on the subdomain interfaces, which results in great 
%problem size reductions, especially for complex geometries with small partition 
%interfaces.
The basis reduction step is
an embarrasingly
parallel batch job that is suitable for
any cluster of networked computing nodes,
including virtual
cloud instances.
Because of no interprocess communication
during the reduction, the bulk of the computation can be done out-of-order
or on cheaper spot instances
that may be interrupted at a moment's notice.
This is in contrast to many of the traditional
iterative solvers based on
domain decomposition; see, e.g.,~\cite{wohlmuth_discretization_2001, toselli_domain_2005, dolean_introduction_2015}.
From a practical perspective,
the technique can be thought of as a more accurate alternative to
the classical component mode synthesis techniques~\cite{roy_craig_coupling_nodate}
that are based on non-overlapping subdomains.
% that may be
% interrupted by the cloud provider at a moment's notice.

After constructing the local bases in parallel, we are then able to
greatly decrease the size of the finite element system matrix 
and improve its conditioning by projecting 
it onto the reduced basis.
A majority of the computational work
can be distributed and only
one layer of interface elements must
be assembled by the master node
where an interpolation technique,
a slight variation of the original partition
of unity method~\cite{babuska_partition_1997},
is used to match the bases originating from the different subdomains.
The global error introduced
by the reduction can be
successfully controlled in the $H^1$ norm
through a parameter which specifies
the maximum tolerance for the local relative
error in each subdomain.
The final reduced linear system can be solved
using traditional techniques such
as the conjugate gradient method on the
master node or, as in our case, on a laptop computer.

The basis reduction is performed using a model order reduction technique that has
been first introduced in the context of multiscale modelling %under the name of MS-GFEM
\cite{babuska_optimal_2011, babuska_machine_2014,
babuska_multiscale-spectral_2020}. The problem domain is partitioned and the
subdomains extended to form an overlap -- the setup is reminiscent of
the classical alternating Schwarz domain decomposition method
\cite{wohlmuth_discretization_2001}. In multiscale modelling, the extended
subdomains are referred to as oversampling
domains~\cite{efendiev_generalized_2014} and an optimal reduced basis can be
constructed by finding and restricting the spectrum of a harmonic extension
from the boundary of each extended subdomain. We note that the harmonic extension is close
to some of the lifting operators used in the iterative domain decomposition
literature; cf.~\cite{toselli_domain_2005}. In the context
of multiscale modelling, we track the utilization of the
oversampling domain trace space back to \cite{efendiev_generalized_2013,
efendiev_generalized_2014}. However, our interest is purely the scaling of the methodology for large-scale computing.

In this study, we solve the elliptic source problem
in a cloud computing environment
and further optimize the computation
through the use of randomized linear algebra~\cite{halko_finding_2011,
martinsson_randomized_2020}, which was applied to model order
reduction in \cite{calo_randomized_2016} and later in
\cite{buhr_randomized_2018, schleus_optimal_2022}.
A technique similar to ours
was used for the distributed solution of the Laplace eigenvalue problem in
\cite{hannukainen_distributed_2022}. 
%Another paper implemented the computations
%via a Steklov eigenvalue problems \cite{ma_novel_2022}. It is worth noting that
%the method is in principle also close to Component Mode Synthesis (CMS)
%methods, although the CMS variants have non-overlapping
%subdomains with nonlocal interface problems for combining the local solutions. 
% Close to the GenEO method?
% \cite{spillane_abstract_2014, dolean_introduction_2015}.
The reduced bases from the local approximation spaces have been shown to 
be optimal in the sense of Kolmogorov $n$-width \cite{babuska_optimal_2011}. 
Precisely, the local approximation error converges exponentially with 
respect to the dimension of the local approximation space. Slightly tighter 
error bounds were derived in \cite{ma_novel_2022} for cuboidal domains, where the
rate of converge was also studied with respect to the  sizes of the subdomain
and its extension. Intuitively speaking, the rate increases with larger subdomain
extensions.

Further, \cite{ma_error_2022} analyzed the discrete problem,
showing similar exponential convergence rates in the finite-dimensional setting
and convergence to the continuous problem as $h\to 0$.
Our work extends the theoretical discrete results. Some eigenvalue problem
bounds were derived in terms of $\ell_2$ norms in
\cite{calo_randomized_2016} whereas here we generate the local reduced bases from
low-rank matrix approximations in the weighted $\ell_2$ matrix norm. The approach is similar to
\cite{buhr_randomized_2018, schleus_optimal_2022} with the difference that we show how to construct the required inner product matrices for properly weighted local norms, and compute the local approximation bases using them
so that the matrix norms and the continuous $H^1$ norms coincide.
%This provides a straightforward and intuitive blueprint for numerical
%implementation.
Convergence to the traditional finite element solution is immediate and the
error or, alternatively, the size of the reduced basis, can be controlled easily using
a single parameter. 

While the existing literature on local approximation spaces often mentions the embarrassingly
parallel nature of the basis creation, \red{only
	\cite{hannukainen_distributed_2022}, focusing
	on the Laplace eigenvalue problem, has previously included larger numerical
examples with unstructured meshes and up to 10 million degrees-of-freedom on a cluster of workstations.}
For very large problems, there are three main challenges that complicate or even prohibit scaling. 
% First, computing the local reduced bases in the proper norm requires computing inner product matrices, which is expensive if computing the local Schur complement explicitly. We show how to do this efficiently without forming the Schur complement. 
First, naively computing the local inner product matrices required for constructing the reduced bases in the proper norm is expensive. We show how to do this efficiently without forming the local Schur complement. 
Second, the local reduced bases need to be as small as possible to create a compact global transformation. Computing the basis with an explicit SVD is costly and using plain sketching from randomized linear algebra \cite{martinsson_randomized_2020}, while fast, introduces lots of redundant dimensions to the basis. We outline a variant of randomized SVD that produces almost optimal local bases with very high probability while being practically as fast as the sketching-only routine. 
Third, the transformation of the original system to the reduced basis requires prohibitive amounts of memory if done naively. We propose a memory-efficient methodology to do the projection such that laptop memory suffices. With these optimizations, the method can be scaled to very large problems in three dimensions.

We focus on the Poisson source problem and compute three-dimensional numerical examples with up to 85
million degrees-of-freedom on a laptop computer by distributing the bulk of the model order reduction to the cloud. The method enables large-scale computing in a publicly
available massively parallel computing environment.  Unlike most studies
that concentrate on structured and regular grid-like mesh partitions and
extensions, we experiment with arbitrary partitions obtained via
graph partitioning tools.
Contrary to the remarks of other authors \cite{calo_randomized_2016, buhr_randomized_2018, schleus_optimal_2022}, the scaling of the local problems with
respect to the number of degrees-of-freedom on the boundary of the
extended subdomain is not an issue in the
distributed setting even when explicitly creating the lifting operator and
finding its singular values. For extended subdomains with
\numprint{10000}--\numprint{20000} nodes, the local problems take around a
minute to solve for our explicit solver. Randomized
numerical linear algebra is used to cut down the solution times to seconds
for the same local problems, although in principle to a lessened degree
of error control. \red{In practice, the combination of the slack local error estimate and the tight probabilistic
bounds given, e.g., in~\cite{halko_finding_2011} causes us to observe pratically identical results with the explicit and the randomized variant.} 
% The errors seem to be generated mainly from combining the bases.

The rest of the paper is organized as follows. In \cref{sec:modelproblem}, we introduce the elliptic model problem and its (approximate) discrete formulation. Next, we give a high-level presentation of the domain decomposition method in \cref{sec:dd}. \Cref{sec:lowrank} follows with an analysis how to approximate the harmonic extension operator in proper norms. In \cref{sec:analysis}, we derive a global error estimate for the reduced basis approximation. Then, we discuss implementational details in \cref{sec:implementation}. \Cref{sec:numtests} closes off with numerical error and scaling tests.

\section{Model problem}
\label{sec:modelproblem}

%In this section, we lay out the problem of interest and 
%some tools to tackle it.
Let $\Omega \subset \R^d$, $d \in \{2,3\}$,
be a polygonal/polyhedral domain and $f \in L^2(\Omega)$.
We consider the boundary value problem: find $\phi$ such that
\begin{equation}
	\begin{aligned}%{2}
		\label{modelproblem}
		-\Delta \phi &= f \quad && \text{in $\Omega$},\\
		\phi &= 0 \quad && \text{on $\partial \Omega$}.
	\end{aligned}
\end{equation}
The finite element method is based on the weak formulation of the above problem. We analyze \eqref{modelproblem} to retain our focus in developing an efficient numerical method.
More general problems in terms of parameterization and boundary conditions have been studied in \cite{babuska_optimal_2011, babuska_machine_2014, babuska_multiscale-spectral_2020}.

\begin{problem}[Continuous formulation]
\label{prob:weak}
    Find $\phi \in H_0^1(\Omega)$ such that
\begin{equation}
\label{eq:weak}
\int_\Omega \nabla \phi \cdot \nabla v \, dx = \int_\Omega f v \, dx\quad \forall v \in H^1_0(\Omega),
\end{equation}
where $f \in L^2(\Omega)$.
\end{problem}

We introduce a shape-regular finite element triangulation/tetrahedralization $\mathcal{T}$, of maximum diameter $h$, and the corresponding
space of continuous, piecewise linear
finite elements $V \subset H^1_0(\Omega)$.
The discrete weak formulation corresponding to
\eqref{eq:weak} is obtained by replacing $H^1_0(\Omega)$ with its finite-dimensional counterpart.
\begin{problem}[Discrete formulation]
\label{prob:dweak}
Find $u \in V$
such that
\begin{equation}
\label{eq:dweak}
\int_\Omega \nabla u \cdot \nabla v \, dx = \int_\Omega f v \, dx\quad \forall v \in V.
\end{equation}
\end{problem}
The number of elements in $\mathcal{T}$ is assumed to be large
enough for the direct solution of
\eqref{eq:dweak} to be unfeasible.
In the next section, we construct a
reduced space $\widetilde{V} \subset V$ so that $\dim \widetilde{V} \ll \dim V$.

\begin{problem}[Reduced discrete formulation]
\label{prob:reduced}
Find $\widetilde{u} \in \widetilde{V}$
such that
\begin{equation}
\label{eq:rdweak}
\int_\Omega \nabla \widetilde{u} \cdot \nabla v \, dx = \int_\Omega f v \, dx\quad \forall v \in \widetilde{V}.
\end{equation}
\end{problem}

In the rest of the paper, we use the following standard notation
for the $L^2$ and $H^1$ norms
over $S \subset \Omega$, respectively:
\[
\| w \|^2_{0,S} = \int_S w^2 \, dx, \quad \| w \|^2_{1,S} = \| w \|_{0,S}^2 + \| \nabla w \|_{0,S}^2.
\]

\section{Domain decomposition}

\label{sec:dd}

The domain $\Omega$
is split into $n$ overlapping subdomains $\Omega = \omega_1 \cup \dots \cup \omega_n$
so that the boundaries $\partial \omega_i$
do not cut through any elements in $\mathcal{T}$ and so
that the intersection of any two subdomains is either empty or a single layer of
elements with thickness $\propto h$, see \cref{fig:decomposition}. 
This decomposition naturally introduces $n$ local meshes
$\mathcal{T}_i$ and finite element spaces $V_i$, $i=1,\dots,n$, where
\[
V_i = \{ w \in H^1(\omega_i) : w|_{\partial \Omega} = 0,\,w|_T \in P_1(T)\,\forall T \in \Th_i\}.
\]

Each local mesh is then extended by adding
all the elements of $\mathcal{T}$ which have any points below distance $r > 0$, i.e.
\[
\Th_i^+ = \{ T \in \Th : \inf_{x \in T, y \in \omega_i} |x - y| < r \}, %\left\{ \boldsymbol{x} \in \Omega : \bs{x} \in T \text{ for some $T \in \mathcal{T}$ and~} \exists \bs{y} \in T \text{~s.t.} \inf_{\bs{z} \in \omega_i} | \bs{y} - \bs{z} | < r  \right\}
\]
where $|\cdot|$ denotes the Euclidean norm.
The corresponding extended subdomain
is denoted by $\omega_i^+$.
Consequently, the boundary $\partial \omega_i^+$
will not cut through any elements and,
therefore,
also the extended subdomains naturally define
$n$ finite element spaces $V_i^+$, $i=1,\dots,n$, where
\[
V_i^+ = \{ w \in H^1(\omega_i^+) : w|_{\partial \Omega} = 0,\,w|_T \in P_1(T)\,\forall T \in \Th_i^+\}.
\]
As demonstrated later,
the choice of the extension size $r$ is a trade-off
between the size of the reduced linear
system and the size of the local
problems that are solved during
the basis reduction process.

\begin{figure}[h!]
    \centering
    \includegraphics[width=0.6\textwidth]{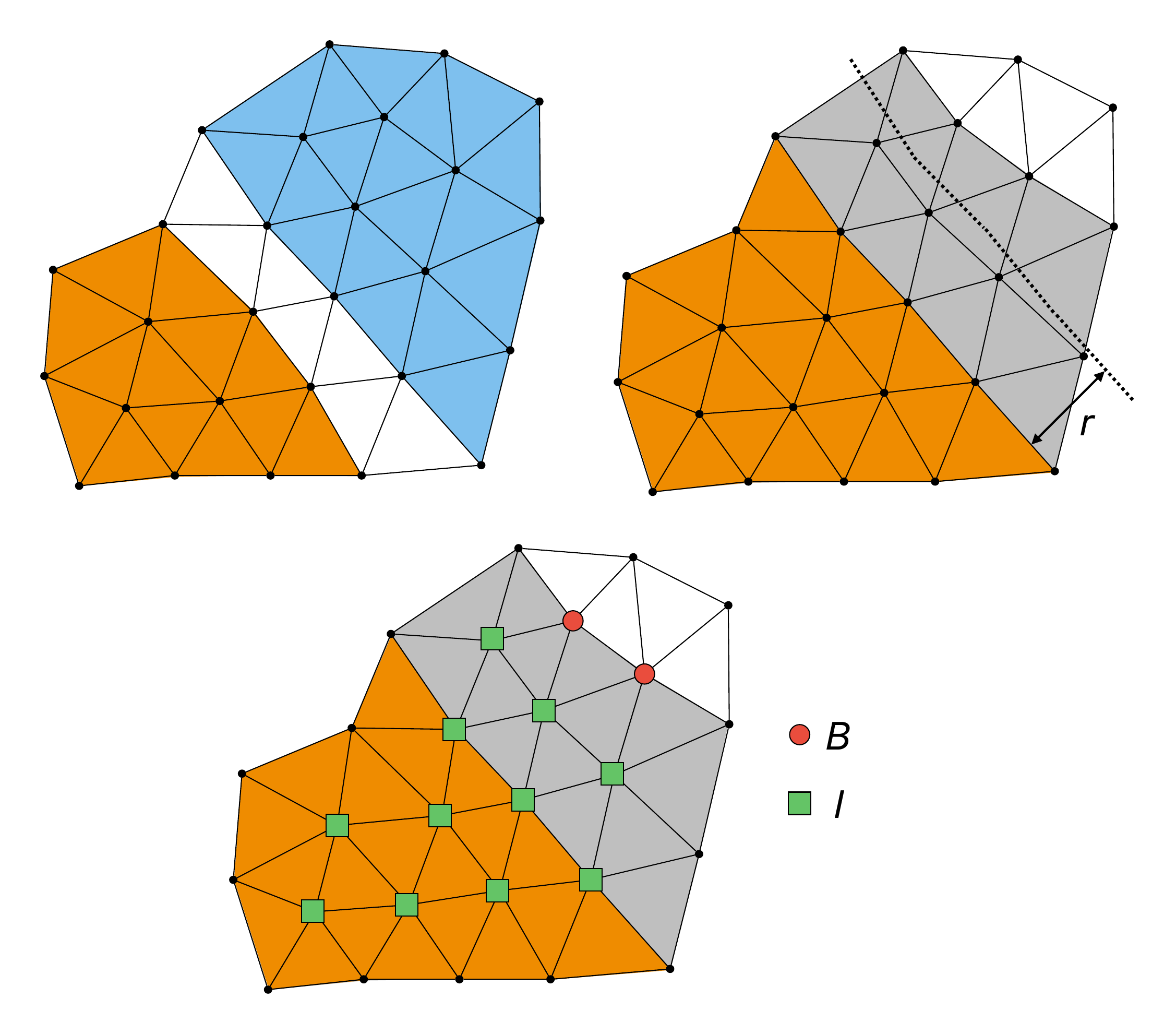}
	\caption{An example with two subdomains. (Top left.) A polygonal domain is split into overlapping subdomains
		$\omega_1$ (orange and white) and $\omega_2$ (blue and white).  The subdomain $\omega_0$
		(white), consists of elements that belong to \textit{both} subdomains -- in general, to more than one subdomain.
		Such a decomposition can be obtained by feeding the mesh edge
		connectivity to a graph partitioning software, e.g., METIS~\cite{karypis_metis_1997} or Scotch~\cite{pellegrini_scotch_1996},
		and finding elements that have at least one node belonging to a
		specific subgraph. (Top right.) The subdomain $\omega_1$ is extended by the
		distance $r$ to define $\omega_{1}^+$ which then consists of the orange
		and the grey elements. Larger $r$
		will decrease
		the size of the global reduced linear system
		while increasing the size of the
		local basis reduction problem.
        (Bottom.) The index sets $I$ and $B$ are used in the
        definition of a discrete lifting operator
	from $\partial \omega_1^+$ to $\omega_1$ in \cref{sec:lowrank}.}
	\label{fig:decomposition}
\end{figure}

A subspace of the above finite element space
with Dirichlet data on the boundary $\partial \omega_i^+ \setminus \partial \Omega$
is denoted by
\[
V_{i,z}^+ = \{ w \in V_i^+ : w|_{\partial \omega_i^+ \setminus \partial \Omega} = z\}.
\]
In the following, we use both the space $V_{i,z}^+$ with an arbitrary boundary condition and $V_{i,0}^+$ with zero boundary. Moreover, we denote the trace space by
\begin{equation}
\partial V_i^+ = \{ w \in L^2(\partial \omega_i^+) : \exists \xi \in V~\text{s.t.}~\xi|_{\partial \omega_i^+} = w \}.
\end{equation}
Consequently, we can define the following subproblem for each extended subdomain $\omega_i^+$:
\begin{problem}[Subproblem]
\label{prob:aux}
\red{Given $z\in \partial V_i^+$}, find $u_i \in V_{i,z}^+$ such that
\begin{equation}
    \label{subproblem}
    \int_{\omega_{i}^+} \nabla u_i \cdot \nabla v \,dx = \int_{\omega_{i}^+} fv\,dx \quad \forall v \in V_{i,0}^+.
\end{equation}
%\begin{equation}
%	\begin{aligned}%{2}
%		\label{subproblem}
%		-\Delta u_i &= f \quad && \text{in $\omega_{r,i}$},\\
%		u_i &= z \quad && \text{on $\partial \omega_{r,i}$,}
%	\end{aligned}
%\end{equation}
\end{problem}
%where $z$ could be set equal to the trace of the solution on the neighboring subdomains to recover the solution
%of \eqref{eq:dweak} from \eqref{subproblem}.
Similarly to \cite{babuska_optimal_2011}, we use \cref{prob:aux}
to construct a reduced basis within the subdomain $\omega_i$
and, finally, combine the bases from the different subdomains in order to
solve \cref{prob:reduced}.
The solution to \cref{prob:aux} can be further split as $u_i = u_{i,0} + u_{i,z}$ where $u_{i,0}$ and $u_{i,z}$ are solutions to the following two auxiliary problems:

\begin{problem}[Subproblem with zero boundary]
\label{prob:aux1}
Find $u_{i,0} \in V_{i,0}^+$ such that
\begin{equation}
    \label{trivialbc}
    \int_{\omega_{i}^+} \nabla u_{i,0} \cdot \nabla v \,dx = \int_{\omega_{i}^+} fv\,dx \quad \forall v \in V_{i,0}^+.
\end{equation}
%\begin{equation}
%	\begin{aligned}%{2}
%		\label{trivialbc}
%		-\Delta u_{i,0} &= f \quad && \text{in $\omega_{r,i}$},\\
%		u_{i,0} &= 0 \quad && \text{on $\partial \omega_{r,i}$},
%	\end{aligned}
%\end{equation}
\end{problem}

\begin{problem}[Subproblem with zero load]
\label{prob:aux2}
\red{Given $z\in \partial V_i^+$}, find $u_{i,z} \in V_{i,z}^+$ such that
\begin{equation}
    \label{nontrivialbc}
    \int_{\omega_{i}^+} \nabla u_{i,z} \cdot \nabla v \,dx = 0 \quad \forall v \in V_{i,0}^+.
\end{equation}
%\begin{equation}
%	\begin{aligned}%{2}
%		\label{nontrivialbc}
%		-\Delta u_{i,z} &= 0 \quad && \text{in $\omega_{r,i}$},\\
%		u_{i,z} &= z \quad && \text{on $\partial \omega_{r,i}$}.
%	\end{aligned}
%\end{equation}
\end{problem}

\Cref{prob:aux2} defines a linear \emph{lifting operator} $\Zh_i : \partial V_{i}^+ \rightarrow V_i, z \mapsto u_{i,z}|_{\omega_i}$.
%$\partial V_{i}^+$ denotes the range of $V_{i}^+$ by the trace operator onto $\partial \omega_i^+ \setminus \partial \Omega$.
Notice that the lifting operator restricts the solution of \cref{prob:aux2} to $\omega_i$.
%\red{$P$ seems to be the most used letter in the literature for our $\Zh$.}
%\begin{definition}[Lifting operator $\Zh_i$]
%\label{def:lifting}
%    Let $z \in \partial V_{r,i}$.
%    Define
%    $\Zh_i z = w|_{\omega_i} \in V_i$
%    where $w$ is the solution
%    to: find $w \in V_{i,z}$ such that
%    \[
%    \int_{\omega_{r,i}} \nabla w \cdot \nabla v \, dx = 0 \quad \forall v \in V_{i,z}.
%    \]
%\end{definition}
The operator coincides with the transfer operator in \cite{buhr_randomized_2018}.

Let $M_i = \dim \partial V_{i}^+$ and $\Set{\varphi_{i,j}}_{j=1}^{M_i}$ be a basis for $\partial V_i^+$.
Using the solutions of \cref{prob:aux1} and \cref{prob:aux2}, due to the additive splitting $u_i = u_{i,0} + u_{i,z}$ and linearity of $\Zh_i$,
for each $z = \sum_{j=1}^{M_i} b_{i,j} \varphi_{i,j} \in \partial V_{i}^+$ it holds
\[
u_i|_{\omega_i} = u_{i,0}|_{\omega_i} + u_{i,z}|_{\omega_i} = u_{i,0}|_{\omega_i} + \Zh_i z = u_{i,0}|_{\omega_i} + \sum_{j=1}^{M_i} b_{i,j} \Zh_i \varphi_{i,j}
.\]
% where $\varphi_{i,j}$, $j=1,\dots,M_i$,
% is a basis for $\partial V_{i}^+$.
This implies that the basis
\begin{equation}
\label{eq:fullbasis}
\{ u_{i,0}|_{\omega_i}, \Zh_i \varphi_{i,1}, \dots, \Zh_i \varphi_{i,M_i} \}
\end{equation}
can represent exactly the solution
of \cref{prob:aux} for any $z \in \partial V_i^+$.
Since the trace of the original solution to 
\cref{prob:dweak}
belongs to $\partial V_i^+$,
it can be also represented by the above basis.

Computing the basis \eqref{eq:fullbasis} requires solving
$M_i + 1$ discrete problems in $V_{i}^+$.
Unfortunately, the dimension $M_i$ can be substantial for larger three-dimensional problems, e.g., for piecewise linear elements $M_i$ equals to the number of nodes on the interface $\partial \omega_{i}^+ \setminus \partial \Omega$.
Therefore, it is better to avoid computing
and using all of the basis functions in \eqref{eq:fullbasis}
by performing
a low-rank approximation of $\Zh_i$.
Fortunately, the singular values of $\Zh_i$
exhibits almost exponential decay and only a small
subset of its range is necessary to approximate
it. %\footnote{This behaviour is studied numerically in  \cref{sec:numtests}.}.
By performing a low-rank approximation, we
arrive at the reduced basis
\begin{equation}
\label{eq:redbasis}
\{ q_{i,0}, q_{i,1}, \dots, q_{i,m_i} \}, \quad q_{i,j} \in V_i,
\end{equation}
where $q_{i,0} = u_{i,0}|_{\omega_i}$ and
$m_i \ll M_i$.

The details of the low-rank approximation
are important for the analysis and performance
of the presented technique. Hence,
they shall be carefully laid out
in \cref{sec:lowrank}.
Special care is devoted to proving how to 
compute the low-rank approximation in the correct 
norm, which is weighted with square roots of 
certain finite element matrices.
Furthermore, a simple optimization
for computing the low-rank approximation 
with randomized numerical linear algebra
is presented in the end of \cref{sec:implementation}.

After the reduced basis \eqref{eq:redbasis} has been computed for each subdomain
$\omega_i$, $i=1,\dots,n$,
we solve \cref{prob:reduced}
within a
reduced discrete space $\widetilde{V} \subset V$, $\dim \widetilde{V} \ll \dim V$.
The reduced space is obtained by
combining the reduced bases \eqref{eq:redbasis} from
different subdomains,
\begin{equation}
\label{eq:globalsolution}
\widetilde{V} = \{ \textstyle\sum_{i=1}^n \Sh_i w_i : w_i \in \widetilde{V}_i \}, \quad \widetilde{V}_i = \mathrm{span}\,\{q_{i,j}\}_{j=0,\dots,m_i},
\end{equation}
where the \emph{stitching operator} $\Sh_i : V_i \rightarrow V$
makes any discrete functions compatible
on the subdomain interfaces.
This is done by
setting the discrete function to zero on the boundary $\partial \omega_i$
and further extending it by zero on the complement $\Omega \setminus \omega_i$.

\begin{definition}[Stitching operator $\Sh_i$]
\label{def:stitching}
	% Let $\Sh_i: V_i\rightarrow V$, $w_i\in V_i$ and $w\in V$. If $\Sh_i, i=1,\dots, n,$ satisfy $$(\Sh_i w_i)|_{\Omega\setminus \omega_i} = 0\quad\text{and}\quad \sum_{i=1}^n \Sh_i (w|_{\omega_i})=w,$$ they are \textit{stitching operators}.
	Let $w_i\in V_i$. A \textit{stitching operator} $\Sh_i: V_i\rightarrow V$ satisfies
	\begin{align*}
	(\Sh_i w_i)(x) =	\begin{cases}
		w_i(x) &\text{ in } \omega_i\setminus \partial\omega_i,\\
		0 & \text{ otherwise.}
	\end{cases}
	\end{align*}
	for every node $x$ of the mesh $\mathcal{T}$.
\end{definition}

\red{
	In practice, with the one node overlap of the subdomains $\omega_i$, the stitching operator acts as a partition-of-unity style approach in combining the subdomains. This formulation simplifies programming the method significantly as there is no explicit cut-off function to include in the finite element assembly, present in the traditional partition-of-unity method~\cite{babuska_partition_1997}. The advantage is compounded when transforming very large problems into the reduced basis, as evident in \cref{sec:implementation}. A simple one-dimensional example is presented in Figure \ref{fig:stitching}. 
}
\begin{figure}[h]
\centering
    %\includesvg[width=0.5\textwidth]{illustraatio.svg}
    \includegraphics[width=0.5\textwidth]{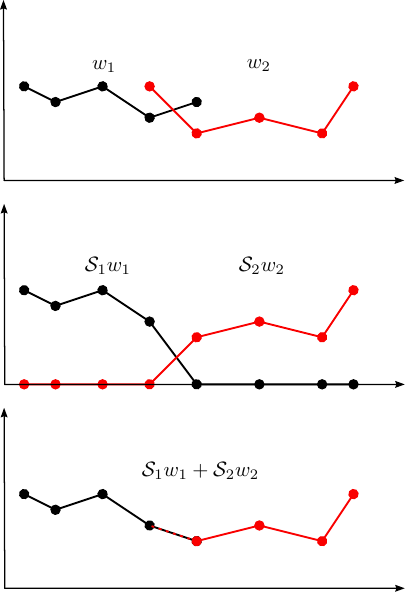}
	\caption{\red{An example of the stitching operator with two finite element functions (red and black) over two subdomains in one dimension.}}
    \label{fig:stitching}
\end{figure}

\section{Low-rank approximation of $\Zh_i$}

\label{sec:lowrank}

The optimal local approximation spaces \eqref{eq:redbasis} can be computed through an eigenvalue problem \cite{babuska_optimal_2011, calo_randomized_2016, hannukainen_distributed_2022} or a low-rank matrix approximation \cite{buhr_randomized_2018, schleus_optimal_2022}. We follow the latter approach, and add to the existing literature how to efficiently compute the low-rank approximation in the correctly weighted $\ell_2$ matrix norm.

It is well known that the best $k$-rank
approximation of a matrix in the
$\ell_2$ norm
is given by its truncated singular
value decomposition consisting of $k$ largest
singular values and the corresponding
singular vectors.
We wish to compute a low-rank
approximation of the lifting
operator $\Zh_i$ in the natural norms
of \cref{prob:aux}, i.e.~the operator norm
\begin{equation}
\label{eq:opnorm}
\|\Zh_i\| = \sup_{z \in \partial V_{i}^+} \frac{\| \Zh_i z \|_{1,\omega_i}}{\| z \|_{\partial V_{i}^+}}
\end{equation}
with the trace norm
\[
\| z \|_{\partial V_{i}^+} := \inf_{\substack{v \in V_{i}^+,\\ v|_{\partial \omega_{i}^+ \setminus \partial \Omega} = z}} \| v\|_{1,\omega_{i}^+}.
\]
The abstract fixed precision low-rank approximation problem is then of the form: given $\epsilon > 0,$ find $\widetilde \Zh_i$ such that
$\| \Zh_i - \widetilde{\Zh}_i \| < \epsilon$ with minimal rank.
% where $\epsilon > 0$.

Next we explain how the low-rank approximation
can be implemented in terms of $\bs{Z}_i$, the matrix
representation of $\Zh_i$, which is given by
\begin{equation*}
%\Zh_i \red{\varphi_{i,j}} = \red{\bs{Z}_i \bm e_j}. 
\Zh_i \red{\varphi_{i,j}} = \sum_{k}q_{i,k}(\bm Z_i \bm e_j)_k. 
\end{equation*}
% where $\varphi_j$, $j=1,\dots,M_i$, is a basis
% for $\partial V_{i}^+$ and
% $\psi_k$, $k=1,\dots,n_i = \dim V_i$, a basis for $V_i$.
Let us denote by $\bm A_i^+$
the stiffness matrix corresponding to \cref{prob:aux}.
The stiffness matrix admits the block structure
\[
	\bm A_i^+ = \begin{bmatrix}\bm A_{II,i}^+ & \bm A_{IB,i}^+ \\ \bm A_{BI,i}^+ & \bm A_{BB,i}^+\end{bmatrix}
\]
where the index set $I$ denotes
the degrees-of-freedom located in the interior of
$\omega_i^+$ and
$B$ denotes the degrees-of-freedom located on the boundary
$\partial \omega_{i}^+ \setminus \partial \Omega$; see \cref{fig:decomposition}.
An alternative expression for $\bs{Z}_i$,
using the stiffness matrix reads
\[
 \bm Z_i = \bm C_i (\bm A_{II,i}^+)^{-1}\bm A_{IB, i}^+,
\]
%\end{equation}
where $\bm C_i \in \mathbb{R}^{n_i \times N_i}$ is a restriction matrix from
the interior nodes of $V_i^+$ to $V_i$ where $N_i = \dim V_i^+ - M_i$.

Now we can represent the operator norm in terms of the matrix $\ell_2$ norm.
Firstly, the numerator within the operator norm \eqref{eq:opnorm} can be written
as
\[
\| \Zh_i z \|_{1,\omega_i} = \| \bs{R}_i \bs{Z}_i \bs{\beta}_z \|_{\ell_2}
\]
where $\bs{R}_i$ is the Cholesky factor of the finite element matrix corresponding to the bilinear form
\[
\int_{\omega_i} (\nabla w \cdot \nabla v + w v)\,dx, \quad w, v \in V_i,
\]
and $\bs{\beta}_z = [b_1, \dots, b_{M_i}]$ is the coefficient
vector corresponding to $z \in \partial V_i^+$.

Secondly, the denominator can be written in terms of $\bs{\beta}_z$, with the help of the following lemmata.
\begin{lemma}
	\label{lem:blockchol}
	Let $\bm A$ be a square symmetric and positive definite matrix with the block structure
 \[
	\bm A = \begin{bmatrix}
    \bm A_{II} & \bm A_{IB} \\
    \bm A_{BI} & \bm A_{BB}
    \end{bmatrix}
\]
 and the Cholesky factor
 \[
 \bm R = \begin{bmatrix}
\bm R_{II} & \bm R_{IB} \\ \bm 0 & \bm R_{BB}
 \end{bmatrix}.
 \]
 Then $\bm A / \bm A_{II} = \bm R_{BB}^T \bm R_{BB}$
 where $\bm A / \bm A_{II} = \bm A_{BB} - \bm A_{BI}\bm A_{II}^{-1}\bm  A_{IB}$
 is the Schur's complement.
\end{lemma}

%\red{Remove proof in the appendix and the reference to it.}
% \begin{proof}
% 	See \cref{proof:blockcholproof}.
% \end{proof}

\begin{lemma}
 \label{lem:boundarytol2norm}
% where $\bm A_{BB,i}^+$ and $\bm M_{BB,i}^+$ correspond
% to the degrees-of-freedom located
% on the boundary $\partial \omega_{i}^+ \setminus \partial \Omega$.
 Choose $\bm A$ in \cref{lem:blockchol} as the finite element matrix corresponding
 to the bilinear form
\[
\int_{\omega_i^+} (\nabla w \cdot \nabla v + w v)\,dx, \quad w, v \in V_i^+.
\]
 It holds
 \begin{align*}
		\|z\|_{\partial V_{i}^+} = \norm[\ell_2]{\bs{R}_{BB,i} \bm\beta_z},
\end{align*}
where $\bm R_{BB,i}$ is the bottom right Cholesky factor of $\bm A$.
\end{lemma}

\begin{proof}
	See \cref{proof:boundarytol2norm}.
\end{proof}

\red{\Cref{lem:boundarytol2norm} is very important in weighing the input in the matrix context. The operator norm \eqref{eq:opnorm} in the discrete setting becomes}
\begin{equation}
\label{eq:dopnorm}
\|\Zh_i\| = \sup_{z \in \partial V_{i}^+} \frac{\| \Zh_i z \|_{1,\omega_i}}{\| z \|_{\partial V_{i}^+}} = \max_{\bs{\beta}_z} \frac{\| \bs{R}_i \bs{Z}_i \bs{\beta}_z \|_{\ell_2}}{\| \bs{R}_{BB,i} \bs{\beta}_z \|_{\ell_2}} = \| \bs{R}_i \bs{Z}_i \bs{R}_{BB,i}^{-1} \|_{\ell_2}.
\end{equation}
\red{The equality \eqref{eq:dopnorm} is crucial, but other authors have computed low-rank approximations without the weights or dismissed them due to difficulties in forming $\bm R_{BB,i}$ and its inverse \cite[Supplementary Material 4]{buhr_randomized_2018}. Without the proper weights, the low-rank matrix approximation is incompatible with the continuous problem.}

We propose that \Cref{lem:blockchol} gives a novel and efficient way to compute $\bm R_{BB,i}$ without forming the Schur complement explicitly. In particular, one computes the Cholesky factor of $\bm A$ and reads the upper triangle matrix $\bm R_{BB}$ from it. The inverse is not required explicitly, and $\bm R_{BB,i}^{-1}$ arises only by solving few systems of equations, see end of \cref{sec:implementation}. Computing $\bm R_{BB,i}$ by explicitly solving the Schur's complement $\bm A/\bm A_{II} = \bm A_{BB} - \bm A_{BI}\bm A_{II}^{-1}\bm A_{IB}$ is much more expensive, which we exhibit in \cref{sec:numtests}. 
In practice, both approaches rely on the Cholesky decomposition, but the latter requires solving increasingly many large systems of equations and an additional matrix multiplication. The proper weight matrices become increasingly hard to compute without our optimization.

As a conclusion, the optimal low-rank
approximation of the operator $\Zh_i$ in the $H^1$ norm
is obtained through the singular value
decomposition of the matrix product
\begin{equation}
\label{eq:svd}
\bs{R}_i \bs{Z}_i \bs{R}_{BB,i}^{-1} = \bs{U}_i \bs{\Lambda}_i \bs{V}_i^T.
\end{equation}
With both $\bs{R}_i$ and $\bs{R}_{BB,i}$ square and full rank, the low-rank
matrix approximation of $\Zh_i$ can be given as the product $\widetilde{\bs{Z}}_i =
\bs{R}_i^{-1} \bs{T}_i \bs{R}_{BB,i}$ for some $\bs{T}_i$ which is obtained by
truncating the singular value decomposition
\eqref{eq:svd}
in the sense that
\begin{equation}
\label{eq:svdtol}
\norm{\Zh_i-\widetilde \Zh_i} = \norm[\ell_2]{\bs{R}_i (\bs{Z}_i - \bs{\widetilde Z}_i)\bs{R}_{BB,i}^{-1}} = \norm[\ell_2]{\bs{R}_i \bs{Z}_i \bs{R}_{BB,i}^{-1} - \bs{T}_i} < \epsilon,
\end{equation}
where $\epsilon > 0$ is a user specified tolerance.
Moreover, the reduced basis $q_{1,i},\dots,q_{m_i,i}$ in \eqref{eq:redbasis} is given by the 
columns of $\bs{R}_i^{-1} \bs{U}_i$ which are directly computable, see 
\cref{sec:implementation}. The number of columns $m_i\ll M_i$ is indirectly
specified by the extension parameter $r$ and the tolerance $\epsilon$.

\section{Error analysis}

\label{sec:analysis}

Next, we derive the error estimate for the reduced discrete problem, \cref{prob:reduced}. The analysis is similar to \cite{babuska_optimal_2011, buhr_randomized_2018}.
The local basis reduction error can be controlled solely with the tolerance
parameter, $\epsilon > 0$, which can be adjusted so that the
global error is dominated by the finite element discretization error.
In the following $\widetilde{V}$ and $\widetilde{V}_i$, $i=1,\dots,n$, are
given by \eqref{eq:globalsolution}.

Let $\phi \in H^1_0(\Omega)$ be the continuous solution to
\cref{prob:weak} and $u \in V$ be the
discrete solution to \cref{prob:dweak}. Further, let $\widetilde{u} \in \widetilde{V}$ be the
reduced solution to \cref{prob:reduced}.
Because \eqref{eq:rdweak} is a Ritz--Galerkin method,
there exists $C>0$ so that 
\begin{equation}
\label{eq:cea}
	\norm[1,\Omega]{\phi-\widetilde{u}} \leq C \norm[1,\Omega]{\phi-\widetilde{v}} \leq C(\norm[1,\Omega]{\phi- u}+\norm[1,\Omega]{u-\widetilde v}) \quad \forall \widetilde{v} \in \widetilde{V}.
\end{equation}
The first term can be estimated using techniques from the standard finite element method a
priori error analysis. Thus, we focus on the second term $\norm[1,\Omega]{u-\widetilde
v}$ and demonstrate how $\widetilde v$ can
be constructed subdomain-by-subdomain to
reveal its
dependency on the tolerance $\epsilon > 0$
of the low-rank approximation.

\subsection{Estimate on a single subdomain}

For the solution $u$
of \cref{prob:dweak} it holds $ u|_{\omega_i} = u_{i,0}|_{\omega_i} + \Zh_i (u|_{\partial \omega_{i}^+})$ where $u_{i,0}$ is the solution of \cref{prob:aux1}.
Let
\begin{equation}
\label{eq:vi}
\widetilde v_i = u_{i,0}|_{\omega_i} + \widetilde\Zh_i (u|_{\partial \omega_{i}^+})
\end{equation}
where $\widetilde\Zh_i : \partial V_{i}^+ \rightarrow \widetilde V_i$ is defined through
the action of the low-rank matrix approximation $\widetilde{\bs{Z}}_i$. Notice that the nontrivial load term $u_{i,0}$ is identical for both $u$ and $\widetilde v_i$.
Therefore, the error can be bounded as
\begin{equation}
	\label{eq:localbound}
	\norm[1,\omega_i]{u-\widetilde v_i} \leq \|\Zh_i - \widetilde \Zh_i\|\norm[\partial V_{i}^+]{u|_{\partial \omega_{i}^+}} \leq \|\Zh_i - \widetilde \Zh_i\| \| u \|_{1,\omega_i^+}.
\end{equation}

Combining \eqref{eq:localbound} with \eqref{eq:svdtol}, we arrive to the following lemma.

\begin{lemma}
\label{lem:localbound}
Let $u \in V$ be the solution
of \cref{prob:dweak} and
$\widetilde{v}_i \in \widetilde{V}_i \subset V_i$ be
defined via \eqref{eq:vi}. It holds
$$
\frac{\norm[1,\omega_i]{ u-\widetilde
v_i}}{\norm[1,\omega_i^+]{ u}} < \epsilon.
$$
\end{lemma}
Thus, 
we interpret the truncation tolerance $\epsilon$
as a local bound for
the relative error caused
by the introduction of the reduced basis.
The bound is in almost all practical cases very crude. This is because the first inequality in \eqref{eq:localbound} moves from a specific boundary trace of $u$ and $\widetilde v_i$ to the operator norm $\norm{\Zh_i - \widetilde \Zh_i}$, which has to account for \textit{all possible boundary conditions}, including very pathological ones. This means that the method tends to be locally considerably more accurate than the given tolerance $\epsilon$, especially when a relatively smooth load $f$ is used.
% It is shown next that
% it is also a global bound for
% the error of the reduced problem \eqref{eq:rdweak}.

\subsection{Full error estimate}

Any discrete function $v \in V$ can be split as
\[
v = \sum_{i=1}^n \Sh_i(v|_{\omega_i}),
\]
where the stitching operator $\Sh_i$
is given by \cref{def:stitching}.
Moreover, all functions in the reduced space $\widetilde{V}$
have the above form by definition.
Hence, we will first show
an asymptotic bound for the action
of $\Sh_i$
and then present the full error estimate which
combines the result of \cref{lem:localbound}
from all subdomains.

\begin{lemma}
\label{lem:stitchbound}
    There exists $C>0$ such that
    \[
    \| \Sh_i \| \leq C(1+ h^{-1}).
    \]
\end{lemma}

\begin{proof}
	The operator norm is
	\begin{align*}
		\norm{\Sh_i} & = \sup_{w \in V_i} \frac{\norm[1,\Omega]{\Sh_i w}}{\norm[1,\omega_i]{w}}.
	\end{align*}
	We start by bounding the numerator. There exists $C_{inv}>0$, independent of $h$  \cite{verfurth_posteriori_2013}, such that 
	\begin{align}
		\label{eq:sbound1}
		\norm[1,\Omega]{\Sh_i w}^2 = \norm[0,\Omega]{\Sh_i w}^2 + \norm[0,\Omega]{\nabla\Sh_i w}^2 \leq (1 + C_{inv} h^{-2})\norm[0,\Omega]{\Sh_i w}^2.
	\end{align}
	Notice that $\text{supp}(\Sh_i w)\subset \omega_i$ so $\norm[0,\Omega]{\Sh_i w}\leq \norm[0,\omega_i]{w}$.

	Now, the denominator is bounded below by
	$$\norm[0,\omega_i]{w}^2 \leq \norm[0,\omega_i]{w}^2 + \norm[0,\omega_i]{\nabla w}^2 = 
	\norm[1,\omega_i]{w}^2,$$ thus 
	\begin{align*}
		\norm{\Sh_i}\leq \frac{\sqrt{1+C_{inv} h^{-2}}\norm[0,\omega_i]{w}}{\norm[0,\omega_i]{w}} \leq C (1 + h^{-1}).
	\end{align*}

\end{proof}

\begin{remark}
	The $h^{-1}$ dependency fortunately manifests only in few pathological cases, and even in those mostly locally. This dependency could be circumvented, e.g., by using
the Lagrange multiplier method or Nitsche's method for combining the local solutions. However, it will be present via the derivative of the cut-off function also when using the standard partition of unity approach.
\end{remark}

\begin{assumption}[Partition overlap]
	\label{ass:overlap}
	Let the triangulation/tetrahedralization $\Th$ be partitioned to $n$ overlapping local meshes $\Th_{i}^+$. \red{Let $m\in\N, m < n,$ such that any element $T\in\Th$ is included in at most $m$ local meshes.} 
\end{assumption}

\Cref{ass:overlap} is satisfied when the extension parameter $r$ is
kept within reasonable limits, such as under half of the subdomain diameter.
Then, $m$ will be a very small integer as any elements are included only in the
intersection of the nearest few subdomains. The number of subdomains $n$ can,
however, grow without influencing $m$. Hence, for relevant problems \red{and decompositions} $m\ll n$.

We can use \cref{ass:overlap} to connect the bound
\eqref{eq:localbound} more effectively with a global error bound. The issue
with \eqref{eq:localbound} is that the norm is transformed from $V_i$ to
$V_i^+$.
Given \cref{ass:overlap}, it is immediate that $$\sum_{i=1}^n \norm[1,\omega_i^+]{u} \leq
m\norm[1,\Omega]{u}.$$ This is also a loose bound because most elements only belong to
individual subdomains.
We are now ready to present the global error estimate.

\begin{theorem}
	\label{thm:globalbound}
	%Assume that $f \in L^2(\Omega)$ and $\phi \in H^2(\Omega)$. Further, 
 Let the finite element mesh partition satisfy \cref{ass:overlap}. Then there exists $C > 0$ such that
	\begin{align}
		\label{eq:globalbound}
		\| \phi - \widetilde{u} \|_{1,\Omega} \leq C(\| \phi - u \|_{1,\Omega} + m (1 + h^{-1}) \epsilon \| u \|_{1,\Omega}).
	\end{align}
\end{theorem}

\begin{proof}
	As in \eqref{eq:cea}, the error can be split into standard finite element error and the reduction error 
	\begin{equation}
 \label{eq:bestapp2}
		\norm[1,\Omega]{\phi-\widetilde{u}} \leq C \norm[1,\Omega]{\phi-\widetilde{v}} \leq C(\norm[1,\Omega]{\phi- u}+\norm[1,\Omega]{u-\widetilde v}) \quad \forall \widetilde{v} \in \widetilde{V}.
	\end{equation}
	Note that the standard finite element error is proportional to $h\norm[0,\Omega]{f}$ \cite{brenner_mathematical_2008}.

    The finite element solution $u$ can be written as
\[
u = \sum_{i=1}^n \Sh_i(u|_{\omega_i}).
\]
Therefore, choosing $\widetilde{v} \in \widetilde{V}$ in the best approximation
result \eqref{eq:bestapp2} as
\[
\widetilde{v} = \sum_{i=1}^n \Sh_i(\widetilde{v}_i),
\]
where $\widetilde{v}_i$ is
given by the expression \eqref{eq:vi}, leads to
\begin{equation}
\label{eq:errorbound}
\begin{aligned}
	\| u - \widetilde{v} \|_{1,\Omega} & \leq \sum_{i=1}^n \| \Sh_i \| \| u|_{\omega_i} - \widetilde{v}_i \|_{1,\omega_i} \\
							  & \leq \sum_{i=1}^n C(1 + h^{-1}) \| u - \widetilde{v}_i \|_{1,\omega_i} \\
							  & \leq \sum_{i=1}^n C(1 + h^{-1}) \epsilon \norm[1,\omega_i^+]{u} \\
							  & \leq C(1 + h^{-1}) \epsilon m\norm[1,\Omega]{u}.
\end{aligned}
\end{equation}
Combining the above with \eqref{eq:bestapp2} concludes the proof.
\end{proof}

\Cref{thm:globalbound} provides a  global error bound for our
method. The error is a combination of the conventional finite element error and the
error introduced by the
projection onto a smaller solution space which is proportional to the parameter $\epsilon$. Further, the more challenging $h^{-1}$
dependency
is due to our stitching operator $\Sh_i$ that luckily manifests only in few pathological cases, and even in those mostly locally. This dependency could be circumvented, e.g., by using
the Lagrange multiplier method or Nitsche's method for combining the local solutions\red{, but it will be present via the derivative of the cut-off function also when using the standard partition of unity approach}.
In practice, the
reduction error is negligible even for very relaxed parameterization, primarily
due to the pessimistic bound \eqref{eq:localbound}. The method can be easily
tuned such that the finite element error is the limiting term, see the numerical results in \cref{sec:numtests}.

\section{Parallel implementation}

\label{sec:implementation}

In this section, we propose an outline to implement our methodology in a setting
with one master node and a host of worker nodes. Given a mesh $\Th$, number of
subdomains $n$, tolerance $\epsilon$, and optionally a rule to define the
subdomain extension parameter $r$, the program outputs a solution to
\cref{prob:reduced} with smaller than $\epsilon$ local relative error compared to
\cref{prob:weak} as per \cref{lem:localbound}.
The full error to the theoretical solution of \cref{prob:weak} is bounded by \cref{thm:globalbound}.

The computational
framework follows a three-step structure: 
\begin{enumerate}
	\item preprocessing the data into extended submeshes;
	\item computing the optimal local bases; and
	\item transforming the global system and solving the smaller problem. 
\end{enumerate}
Steps (1) and (3) can be computed on the master node while step (2) can be
completely parallelized in a distributed setting without node communication; see \cref{fig:cloud}.
For testing purposes, step (2) may be run sequentially on a single computer.
Code corresponding to the framework using
\texttt{scikit-fem} \cite{gustafsson_scikit-fem_2020} can be found
from \cite{sourcepackage}. This section omits some optimizations included in the code for clarity. In
\cref{sec:numtests}, the scaling tests suggest approximate
hardware requirements for problems of different sizes.
We will now address steps (1)-(3) in order, but focus primarily on step (3) because the
first two are direct applications of \cref{sec:dd} and
\ref{sec:lowrank}.

\subsection*{Step (1)}

Beginning with preprocessing, the connectivity between the
mesh nodes is interpreted as a graph and the nodes are split into $n$
overlapping subdomains using a graph partitioning algorithm~\cite{karypis_metis_1997, pellegrini_scotch_1996} according to \cref{fig:decomposition}.
%If the partition produces
%non-overlapping subdomains, an efficient data structure to enlargen the
%subdomains is the $k$-d tree \cite{bentley_multidimensional_1975}. 
The submeshes $\Th_{i}^+$, which include information of both the subdomains~
$\omega_i$ and the extended subdomains $\omega_{i}^+$, and their neighboring
relations are saved to files. This
concludes step~(1).

\begin{figure}
\centering
    \includegraphics[width=0.8\textwidth]{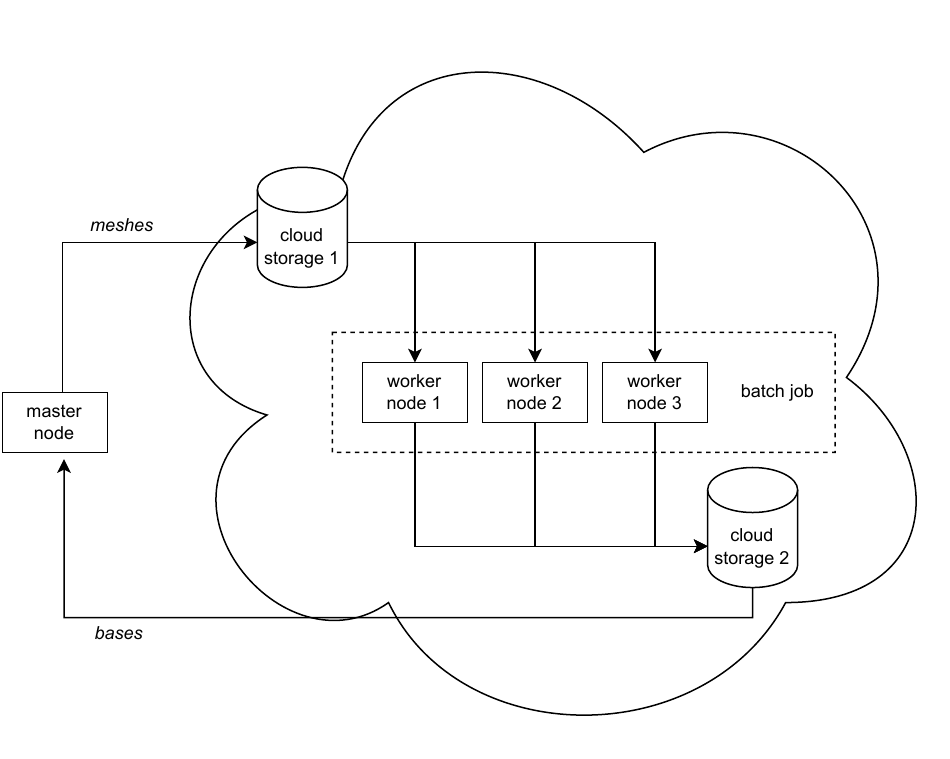}
    \caption{A high level overview of the data flow in our distributed implementation of step (2).
    Decomposing the mesh
    into overlapping subdomains
    is done by the master node in step (1).
    The overlapping meshes are then uploaded
    to a cloud storage bucket
    where the worker nodes 
    are able to access the files.  The worker nodes
    are virtual machines launched
    as a batch job.
    The figure depicts an example with three worker nodes/subdomains.
    The worker nodes write their output
    data to a separate cloud storage
    bucket which can be accessed by
    the master node in order to proceed with
    step~(3).}
    \label{fig:cloud}
\end{figure}

\subsection*{Step (2)}

The files are sent to worker nodes for computing the reduced local
bases. Before going into computations, it is worth to describe briefly how the
local bases are used to transform the system in the end. Let $\bm Q_i$ denote
the reduced local basis corresponding to $\omega_i$ and $\bm A$ the global stiffness matrix
corresponding to \cref{eq:dweak}. 
\red{In step (3), we would like to project the original system for}
\begin{equation}
	\begin{aligned}
	\label{eq:qtaq}
	\bm Q^T \bm A\bm Q & = \mat{\bm Q_1, , \bm 0; , \ddots, ; \bm 0, , \bm Q_n}^T \mat{\bm A_{11}, \cdots, \bm A_{1n}; \vdots, \ddots, \vdots; \bm A_{n1}, \cdots, \bm A_{nn}} \mat{\bm Q_1, , \bm 0; , \ddots, ; \bm 0, , \bm Q_n} \\ 
					   & = \mat{\bm Q_1^T\bm A_{11}\bm Q_1, \cdots, \bm Q_1^T\bm A_{1n}\bm Q_n; \vdots, \ddots, \vdots; \bm Q_n^T\bm A_{n1}\bm Q_1, \cdots, \bm Q_n^T\bm A_{nn}\bm Q_n}.
	\end{aligned}
\end{equation}
\red{Hence, in step (2), we compute $\bm Q_i$ and the product $\bm Q_i^T \bm A_{ii}\bm Q_i$ and send those to the master node.}

\red{We proceed first by assembling a local system from the local mesh
corresponding to $\omega_{i}^+$.} Then, creating the optimal reduced basis for
$\omega_i$ proceeds again in two substeps. Firstly, \cref{prob:aux1} is
solved and restricted to $\omega_i$ to obtain the first basis function.
Secondly, the rest of the basis is obtained by computing the matrix-version of
the lifting operator $\Zh_i$ as
\begin{align*}
	%\label{eq:zprod}
	\bm Z_i = \bm C_i (\bm A_{II,i}^+)^{-1}\bm A_{IB, i}^+.
\end{align*} 
Next, the weight matrices $\bm R_i$ and $\bm
R_{BB,i}^+$ are solved following \cref{sec:lowrank}, and the
low-rank approximation computed per \eqref{eq:svdtol}. 
The left singular vectors of the low-rank approximation can be denoted $\bm U_{i,\epsilon}$.
Then, the energy norm
weights are removed 
% from the left singular vectors of the low-rank approximation 
and the range restricted, $\bm Q_i = \bm D_i \bm R_i^{-1}\bm
U_{i,\epsilon}$, where $\bm D_i$ is another restriction matrix from the degrees-of-freedom corresponding to $\omega_i$ to the
degrees-of-freedom correponding to the interior of $\omega_i$. The restriction matrix $\bm D_i$ acts as the stitching operator in \cref{def:stitching}.
The range now corresponds to the correct solution space. 
To improve accuracy and to reduce memory usage and transfer, a generalized eigenvalue problem is still solved for a slightly updated basis. This optimization diagonalizes the diagonal blocks of \eqref{eq:qtaq}. The details can be inferred from our source code \cite{sourcepackage}.

Step (2) is finished by sending the relevant information back to the master
node. This includes the bases $\bm Q_i$, the projected local problems and
some indexing information of $\omega_i$ and $\partial\omega_{i}^+$.
This data is sufficient for transforming the global system. 

\subsection*{Step (3)}

The challenge in step (3) is in the transformation \eqref{eq:qtaq} for
extremely large $\dim V$. While the global stiffness matrix $\bm A$ is very
sparse, the number of degrees-of-freedom makes direct computation of $\bm Q^T \bm A \bm Q$ infeasible. Fortunately, the matrix product \eqref{eq:qtaq}
can be constructed implicitly, assembling neither $\bm A$ nor
$\bm Q$. In particular, the diagonal submatrices were already computed in step (2).
Further, the vast majority of the off-diagonal submatrices $\bm Q_i^T \bm
A_{ij}\bm Q_j$, which correspond to the nodes in the interiors of
$\omega_i$ and $\omega_j$, $i\neq j$, are predictably zero.

The nonzero off-diagonals can be constructed as follows. Let the set of elements that are included in at least two subdomains be denoted
by $\omega_0 = \cup_{i=1}^n \cup_{j=1}^n (\omega_i\cap\omega_j)$. The nonzero
values correspond to the interior degrees-of-freedom of $\omega_i$ and
$\omega_j$ that influence the same interfacing element in
$\omega_0$, and these values are dependent on \textit{the interfacing
elements only}. Note that each subdomain is neighboring only a
minority of the other subdomains. For non-neighboring subdomains, the
matrix $\bm Q_i^T \bm A_{ij} \bm Q_j$ is zero as $\bm A_{ij}$ is a zero
matrix. For neighboring subdomains, the intersection $\omega_i\cap\omega_j$ is
still very small, which means that $\bm A_{ij}$ is very sparse. Moreover, the
nonzero values are known in advance to correspond to the degrees-of-freedom which
are included in $(\omega_i\cap\omega_j)\setminus\partial\Omega$.

Hence, the off-diagonal submatrices can be also computed in
parallel with respect to a stiffness matrix $\bm A_{0}$ assembled over the set of interfacing elements $\omega_0$ and indexed on the
respective boundary degrees-of-freedom. This is a much cheaper procedure: $\bm A_{0}$ is
\textit{extremely} sparse and hence fits into memory easily \red{even for extremely
large problems}. Moreover, it
becomes evident that we do not need the full $\bm Q_i$ from step
(2) but only the matrix rows corresponding to the boundary degrees-of-freedom. 
The off-diagonal products are then with respect to a small subset of the
degrees-of-freedom and almost all of the products are predictably zero due to
the existence of non-neighboring subdomains.

Projecting the stiffness matrix onto the optimal basis can be now done in parts, by iterating over the pairs of neighboring subdomains in parallel. The projection of the system is then possible to compute on a laptop even for large problems. Without this optimization, scaling the method requires prohibitive amounts of memory. The projection of the right-hand side was
already done in step (2). The reduced global system is then fast to solve using an iterative
method such as the preconditioned conjugate gradient method and a diagonal preconditioner.
% Pseudocode with commentary for steps (2) and (3) can be found from 
% \cref{app:pseudocode}. 

\subsection*{Randomized numerical linear algebra}

The whole procedure can be managed on a laptop master node for considerable problem sizes
if the bulk of the work is distributed to worker nodes in step
(2). Still, constructing $\bm Z_i$ and its low-rank approximation scale poorly in step (2). Luckily, both can be optimized as follows.

The tool we use is \emph{sketching} from randomized numerical linear algebra
\cite{martinsson_randomized_2020}. In a nutshell, the \cite[Algorithm
4.1]{halko_finding_2011} samples a smaller standard Gaussian random matrix $\bm
S_i\in\R^{M_i\times k_i}, k_i < M_i$, mapping it with the weighted $\bm Z_i$,
and performing a $\bm{QR}$ factorization on the smaller matrix product
\begin{align*}
	\label{eq:sketcheq}
	\hat{\bm Q_i} \hat{\bm R_i} = \bm R_i \bm C_i \bm A_{II,i}^{-1} \bm A_{IB,i} \bm R_{BB,i}^{-1}\bm S_i.
\end{align*}
The resulting $\hat{\bm Q_i}$ spans a subspace of the image of the weighted $\bm Z_i$
determined by random boundary conditions. This is similar to \cite{buhr_randomized_2018}, but we have added the required inner product matrices $\bm R_i$ and $\bm R_{BB,i}^{-1}$. The expected and tail errors from using this approach are detailed in \cref{error:lowrank} and we discuss how to choose the sketching parameter at the end of \cref{sec:scalingtests}.

However, compared to doing the SVD of the weighted $\bm Z_i$ directly to find the smallest local approximation space for $\epsilon$, the local spaces from sketching can be much larger than desired due to oversampling. We noticed that even iterative sketching procedures often generated local approximation spaces twice the size of an explicit SVD. The optimal local space is a subspace of $\hat{\bm Q_i}$ with a very high probability, so the SVD could be done the for sketched matrix. However, this produces poor results due to scaling singular values with the random boundary conditions. 

We propose an efficient approach to find the optimal local approximation space by projecting the image of $\bm Z_i$ 
onto $\hat{\bm Q_i}$, and then doing an SVD for the matrix product
\begin{equation}
\label{eq:sketchsvd}
\hat{\bm Q}^T_i\bm R_i \bm C_i \bm A_{II,i}^{-1} \bm A_{IB,i} \bm R_{BB,i}^{-1}.
\end{equation}
Taking the
singular vectors that correspond to singular values greater than tolerance for
$\hat{\bm U}_{i,\epsilon}$, the updated basis is given by $\bm Q_i = \bm D_i\bm R_i^{-1} \hat{\bm Q_i} \hat{\bm U}_{i,\epsilon}$.
This is known as the randomized SVD \cite{martinsson_randomized_2020}. The problem is that computing \eqref{eq:sketchsvd} from right-to-left now requires solving $M_i$ large systems of equations in $A_{II,i}^{-1}A_{IB,i}$ compared to $k_i$ in sketching. We resolve this issue by noting that
by transposing the matrix product \eqref{eq:sketchsvd} twice, it can be computed with two left solves as
\begin{align*}
	\hat{\bm Q}^T_i\bm R_i \bm C_i \bm A_{II,i}^{-1} \bm A_{IB,i} \bm R_{BB,i}^{-1} = (\bm R_{BB,i}^{-T} ((\bm A_{II,i}^{-T}(\hat{\bm Q}^T_i \bm R_i \bm C_i)^T)^T \bm A_{IB,i})^T)^T,
\end{align*}
so that the linear solves are with respect to $k_i$ instead of $M_i$ right-hand sides. This retains the asymptotics of plain sketching and approximately only doubles the hidden coefficient.

Our 
approach produces \textit{practically the same optimal bases as explicitly constructing
weighted $\bm Z_i$ and doing an SVD, with a very high probability}, while being a magnitude faster than the explicit variant for relevant problem sizes. In contrast, creating the reduced basis by sketching only introduces plenty of redundant dimensions to the projection. Achieving the smallest reduced bases is important for very large problems to fit the reduced global problem into memory. Our optimization enjoyes the benefits of both speed and small dimensionality.

Further, we repeat that the bound in \cref{lem:localbound} is very crude, and hence the local error tends to be much smaller than the tolerance $\epsilon$. This is numerically supported in \cref{sec:scalingtests}. As a result, we experienced the explicit and randomized variants in step (2) to produce practically identical global errors in our simulations.

\section{Numerical experiments}

\label{sec:numtests}

We provide three kinds of numerical results to support the theoretical
discussion. First, we introduce scaling tests on a unit cube
with evenly spaced tetrahedral elements and the canonical hat basis,
with the model order reduction implemented in a cloud environment. The numerical error and the use of computational
resources are
analyzed in detail. Secondly, we solve an example of \eqref{modelproblem} with
a nonconstant parameter $a=a(x)$.  Thirdly, we showcase the
methodology in the case of a "$2.5$-dimensional" engineering geometry, which provides
a more realistic application with low-dimensional partition interfaces.

%We now provide numerical validation to the theoretical constructs.
\subsection{Scaling tests}
\label{sec:scalingtests}
Cube is one of the most
challenging objects geometrically for the method as decomposing it into
subdomains produces large interfaces and hence larger optimal reduced bases.
The scaling test included $7$ cases of increasing size up to \numprint{86350888} degrees-of-freedom. 
\Cref{prob:weak} with load
%$$f = \sin(\pi x) \sin(\pi y) \sin(\pi z)$$ on $\Omega = [0,1]^3$ by solving
$$f = 2\sqrt{900}((1-x)x(1-y)y + (1-x)x(1-z)z + (1-y)y(1-z)z)$$ on $\Omega = [0,1]^3$
was approximated by solving
\cref{prob:dweak} using the domain decomposition method. Such a loading was chosen
because the energy norm of the corresponding analytical solution $\phi$ is equal to $1$.
Then by the Galerkin orthogonality, the error can be
computed as
\begin{align*}
	% \label{eq:numerror}
	\norm[0,\Omega]{\nabla(\phi-\widetilde u)} & = (\norm[0,\Omega]{\nabla \phi}^2 - \norm[0,\Omega]{\nabla \widetilde
	u}^2)^{1/2} = (1 - \bm x^T \bm Q^T \bm A \bm Q \bm x)^{1/2},
\end{align*}
where $\bm x$ is the reduced solution vector.

The tolerance was set to $\epsilon=\expnumber{1}{-2}$
which now corresponds to the maximum relative error of 1\,\% 
\red{for each subdomain.}
This is enough even for the largest cases, evident in \cref{tbl:scaling} and \cref{fig:errors}, due to the looseness of the bound in \cref{lem:localbound}.
% The master node was a laptop computer with Intel Core i5-1335U CPU and 32GB memory. 
The number of subdomains was chosen such that the
number of degrees-of-freedom on each subdomain would be
in few thousands. The mesh partitions were produced by METIS
\cite{karypis_metis_1997} up to case 8 and with a custom regular cube partitioning algorithm
afterwards; the bottleneck on the master node was ultimately the memory
usage of the mesh preprocessing.
The extension \red{parameter $r$ was defined via mesh node hops, hereon called
	$r$-hops, including into the extension all nodes within $r$ edges from the
subdomain. The} hops were specified to approximately
double the diameter of the subdomains. In a three-dimensional setting, this
roughly eightfolds the subdomain volume or the number of degrees-of-freedom in a
computational setting. 

The local reduced bases were computed in Google Cloud using the last sketching
procedure detailed at the end of \cref{sec:implementation} with 
$k_i=\lfloor M_i/8 \rfloor$. This choice is justified by the spectra of the local operators, discussed at \cref{app:spectra}. 
The cloud worker nodes were employed with \texttt{c2-standard-4} machine type
and \texttt{debian-11-bullseye-v20230411} image at spot prices for lower costs.
Up to case 8, $n/5$ virtual machines 
% with 4 CPU cores and 4GB memory 
were reserved 
so that each instance computed five local problems. From case~9 and
onwards 1000 instances were reserved. Solving a local problem took in general under a minute of wall clock time,
but due to waiting for resources at spot prices the mean times could increase
to three minutes per problem across the largest jobs.

\begin{table}[H]
	\caption{Computational environment for the scaling tests.}
\label{tbl:compenv}
\centering
\begin{tabular}{l l l c r}
\toprule
Node & OS & CPU & Threads & RAM \\
\midrule
Master & Ubuntu 22.04 LTS & Intel Core i5-1335U & 12 & 32GB \\
Worker & Debian Bookworm 12.5 & Intel Xeon Gold 6254 & 4 & 4GB  \\
\bottomrule
\end{tabular}
\end{table}

The local reduced bases were transferred back to the master node, where 
the global stiffness matrix was transformed using 
optimizations introduced in \cref{sec:implementation}. 
The reduced linear system of equations was finally solved using the conjugate gradient
method with a diagonal preconditioner. Even for case 11, the postprocessing and
the solve took roughly four hours, where the computational time 
concentrated almost exclusively on the system transformation. The master and worker node specifications are detailed in \cref{tbl:compenv}.

The results can be found from \cref{tbl:scaling} and \cref{fig:errors}. The error from the
theoretical solution decreases with $h$ as expected. The reduction error, i.e.~the error between the conventional FE solution and our domain decomposition
method, stays relatively constant at double the local error bound. The Google
Cloud cost in dollars is mainly from computational resources with some
additional costs from data transfer and storage.
\begin{table}[H]
	\caption{Scaling tests with $\epsilon=\expnumber{1}{-2}$. The reduction error has not been computed for cases 8 and above due to the large sizes of the original systems.}
% The columns are \\ Case // degrees-of-freedoms // degrees-of-freedoms of reduced problem // number of
% subdomains // degrees-of-freedoms of first subproblem // error w.r.t. analytical solution // reduction error w.r.t. direct FEM solver}
\label{tbl:scaling}
\centering
\begin{tabular}{c r c c c c c c l}
\toprule
Case & $\dim(V)$ & $n$ & $r$-hop & $h$ & error & red. error & cost\\
\midrule
4 & \numprint{4233} & \numprint{6} & 4  & $\expnumber{1.5}{-1}$ &  $\expnumber{1.6}{-1}$ &  $\expnumber{9.0}{-5}$  & $\leq$\$0.1 \\
5 & \numprint{30481} & \numprint{30} & 5 & $\expnumber{9.9}{-2}$ & $\expnumber{8.0}{-2}$  &  $\expnumber{2.1}{-4}$ & $\leq$\$0.1 \\
6 & \numprint{230945} & \numprint{100} & 6 & $\expnumber{4.9}{-2}$ & $\expnumber{4.1}{-2}$  &  $\expnumber{2.6}{-4}$ & \$0.16  \\
7 & \numprint{1000000} & \numprint{400} & 6 & $\expnumber{1.7}{-2}$ & $\expnumber{2.1}{-2}$  &  $\expnumber{2.0}{-4}$ & \$0.73  \\
8 & \numprint{10648000} & \numprint{4000} & 6 & $\expnumber{7.9}{-3}$ & $\expnumber{9.5}{-3}$  &  - & \$4.36 \\ %\$6.33  \\
9 & \numprint{21952000} & \numprint{8000} & 6 & $\expnumber{6.2}{-3}$ & $\expnumber{7.5}{-3}$  &  -  & \$8.64 \\
10 & \numprint{42875000} & \numprint{15625} & 6 & $\expnumber{4.9}{-3}$ & $\expnumber{6.0}{-3}$  &  -  & \$15.75 \\
11 & \numprint{86350888} & \numprint{39304} & 6 & $\expnumber{3.9}{-3}$ & $\expnumber{4.7}{-3}$  &  -  & \$39.85 \\
\bottomrule
\end{tabular}
\end{table}

\begin{figure}
	\begin{center}
		\includegraphics[width=0.85\textwidth]{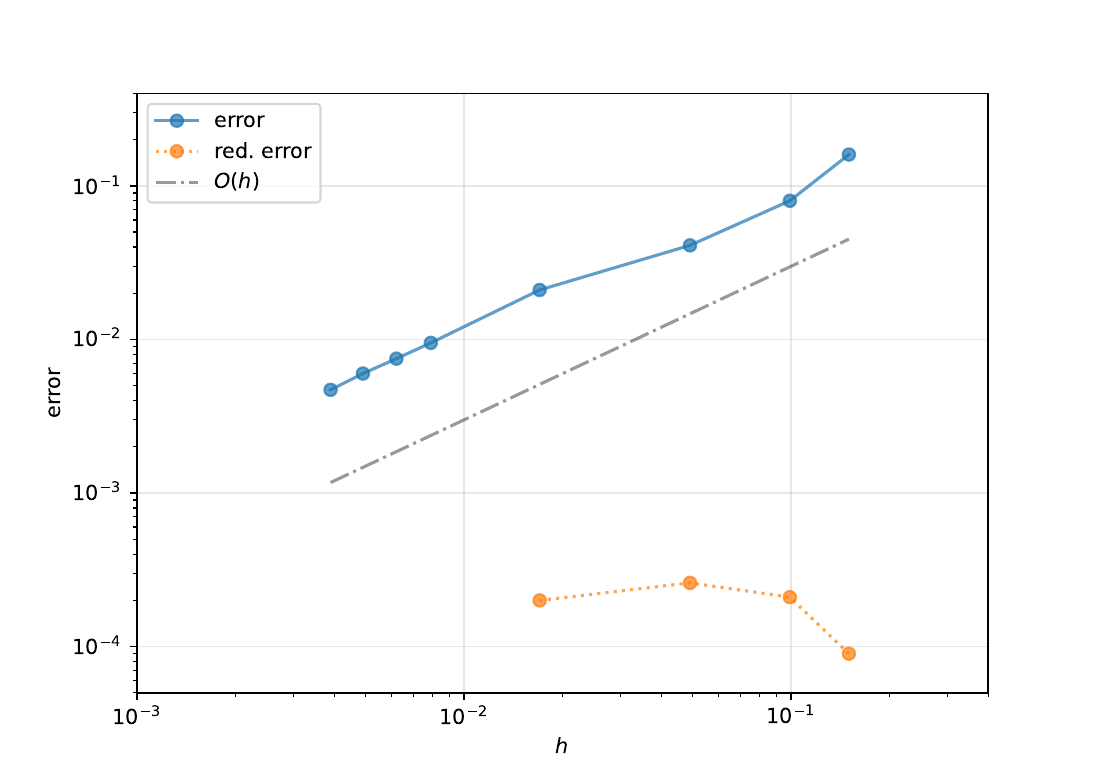}
	\end{center}
	\caption{Convergence of the method on a log-log scale. The $x$-axis depicts the mesh parameter $h$, and the $y$-axis shows the errors $\norm[0, \Omega]{\nabla (\phi - \tilde u)}$, the reduction errors $\norm[0,\Omega]{\nabla (u - \tilde u)}$ and the theoretical FEM convergence rate. The reduced method follows the conventional rate since the reduction error stays well below the FEM error. The reduction errors were not computed for the largest cases due to limitations in scaling the conventional FEM.}\label{fig:errors}
\end{figure}

% eigl = np.abs(sp.linalg.eigs(Kred, which='LM', return_eigenvectors=False))
% eigs = np.abs(sp.linalg.eigs(Kred, sigma=1e-8, return_eigenvectors=False))
% np.max(eigl)/np.min(eigs)

\begin{figure}
	\centering
	\begin{subfigure}[b]{0.49\textwidth}
		\centering
		\includegraphics[width=\textwidth]{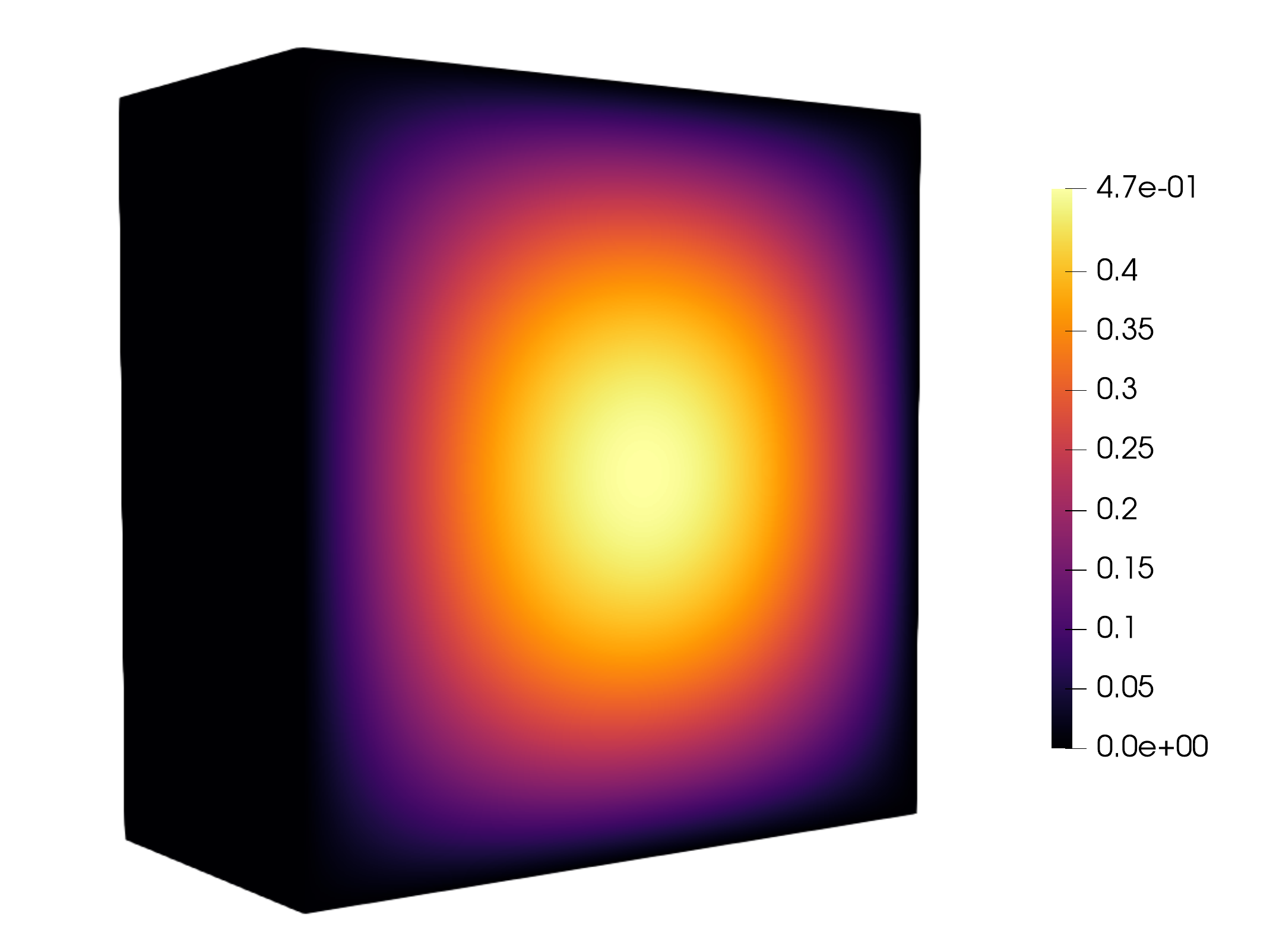}
		\caption[]%
		{{\small Case 7 reduced solution}}
	\end{subfigure}
	\hfill
	\begin{subfigure}[b]{0.49\textwidth}  
		\centering 
		\includegraphics[width=\textwidth]{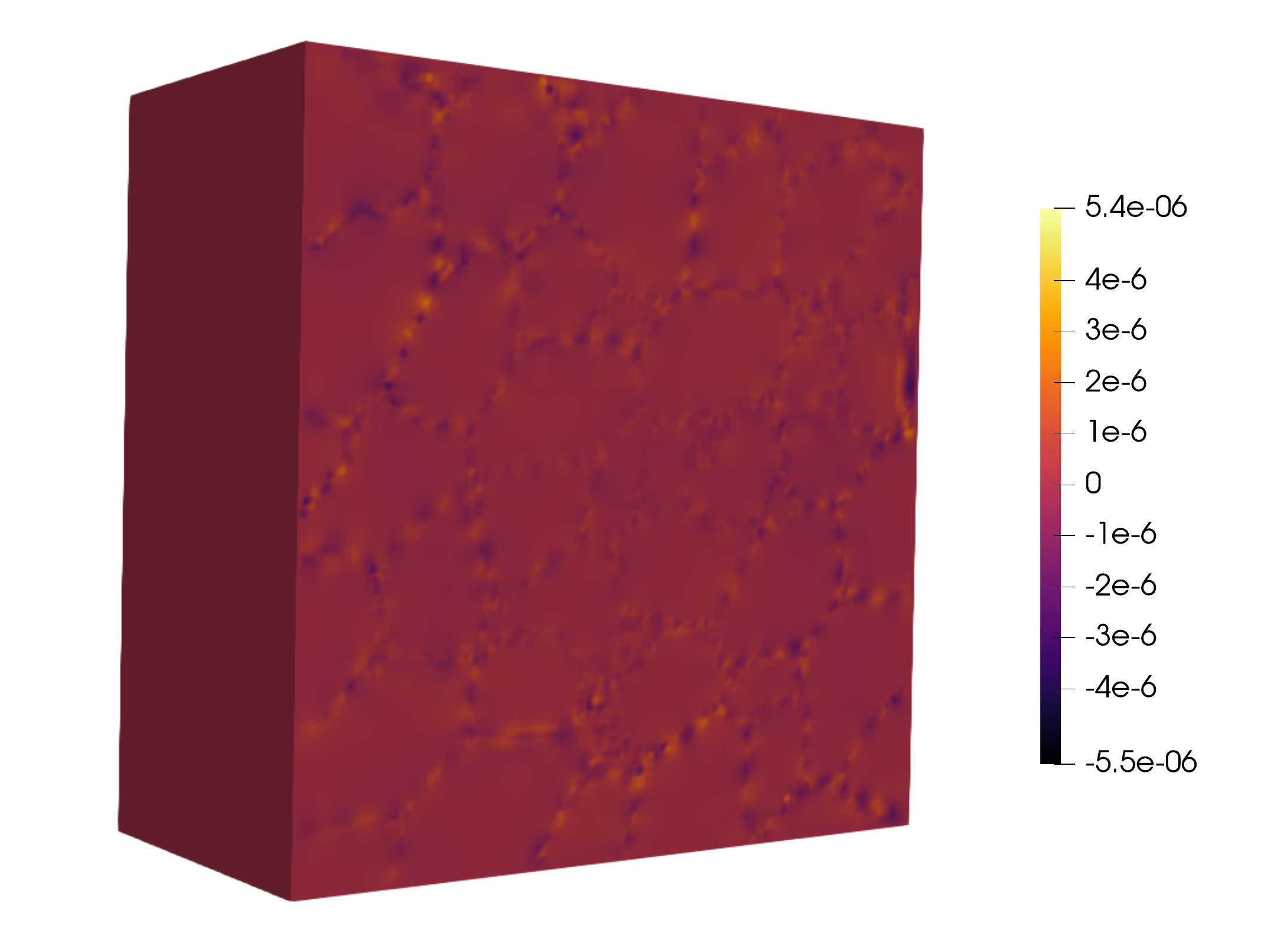}
		\caption[]%
		{{\small Case 7 reduction error $u-\widetilde u$}}
	\end{subfigure}
	\caption[ Case 7 error ]
	{\small Case 7 reduced solution and error. The error is negligible and concentrates on the subdomain interfaces.}
	\label{fig:c7sol}
\end{figure}

\Cref{fig:c7sol} displays the reduced solution and its error compared to the FE solution for case 7. The approximation is practically exact inside the subdomains, but some error persists on the subdomain interfaces. 

\Cref{tbl:nonzeros} presents how the global system size and conditioning is
improved for cases in \cref{tbl:scaling}, although with small increase in the
number of nonzero values. The reductions are significant with the reduced bases having around 
$20-30$ times fewer degrees-of-freedom than the original bases. The condition numbers are also about
half of the original system value. Unfortunately, the relative fill-ins of the
reduced stiffness matrices are much larger because the support of the
reduced bases are larger then the support of the original piecewise
linear finite element bases, and the nonzero off-diagonal
submatrices are dense. The size of the dense blocks can be reduced by further increasing the number
of extension hops. 
However, this cost is negligible compared to the fact that it is infeasible to assemble the original stiffness matrix altogether with the same memory.
% \Cref{fig:c7qtaq} displays the reduced stiffness matrix of case 7.

\begin{table}[H]
	\caption{Dimensions, number of nonzeroes and condition numbers of the original and reduced stiffness matrices for the scaling tests in \cref{tbl:scaling}. The cases 8 and below used METIS in the graph partitioning.  The cases 9 and above used a custom block partitioning algorithm which splits the cube into smaller cubes of equal size. The regular shape of the cubical subdomains and their extensions explains the drop in the number of nonzeros from case 8 to case 9.  The number of nonzeros and the condition numbers have not been computed for the cases 8 and above due to the high computational efforts involved.}
\label{tbl:nonzeros}
\centering
\begin{tabular}{c r r r r r r r }
\toprule
Case & $\dim(V)$ & $\dim(\widetilde V)$ & nnz($\bm A$) & nnz($\widetilde{\bm A}$) & $\kappa(\bm A)$ & $\kappa(\widetilde{\bm A})$ \\
\midrule
4 & \numprint{4233} & \numprint{178} &\numprint{56195} & \numprint{21888} & 34 & 50 \\
5 & \numprint{30481} & \numprint{1436} &\numprint{424627} & \numprint{682710} & 445 & 211 \\
6 & \numprint{230945} & \numprint{6880} &\numprint{3289443} & \numprint{5800292} & 2161 & 816 \\
7 & \numprint{1000000} & \numprint{35041} &\numprint{6940000} & \numprint{42801805} & 3971 & 2272 \\
8 & \numprint{10648000} & \numprint{425194} & - & \numprint{659025526} & - & - \\ % nnz(A) \numprint{74245600}
9 & \numprint{21952000} & \numprint{895208} & - & \numprint{600208952} & - & - \\
10 & \numprint{42875000} & \numprint{1782383} & - & \numprint{1219492727} & - & - \\
11 & \numprint{86350888} & \numprint{4019216} & - & \numprint{2466745712} & - & - \\
\bottomrule
\end{tabular}
\end{table}

% \begin{figure}[h!]
% \centering
%     \includegraphics[width=0.6\textwidth]{c7kred.pdf}
% 	\caption{The esreduced stiffness matrix $\widetilde{\bm A}$ of case 7 in \cref{tbl:nonzeros}.}
%     \label{fig:c7qtaq}
% \end{figure}

\subsection{Subproblem scaling}

We claimed that the method is difficult to scale without optimizing the computation of $\bm R_{BB,i}$, the construction of the optimal reduced basis $\bm Q_i$ and the projection of the system to the reduced basis $\bm Q^T \bm A\bm Q$. Our optimization for the first problem was given in \cref{sec:lowrank} while optimizations for the latter two were delineated in \cref{sec:implementation}. We now present numerical evidence of the subproblem scaling focusing first on the two subproblems that are local and then examining the global transformation of the system to the reduced basis. All computations were completed on the master node in \cref{tbl:compenv}.

\Cref{fig:subproblemscaling} compares the scaling of $\bm R_{BB,i}$ and $\bm Q_i$ by first computing the desired subproblems with a straightforward approach and secondly with our optimizations. For $\bm R_{BB,i}$, explicitly computing the Schur complement tends to scale quadratically in wall clock time compared to sub-quadratic scaling of the Cholesky decomposition for the extended local problem. 
% In practice, both approaches rely on the Cholesky decomposition, but the former requires solving increasingly large systems of equations and an additional matrix multiplication. 
The effect becomes more pronounced when we compute the smallest possible basis $\bm Q_i$ given a tolerance $\epsilon=\expnumber{1}{-2}$. Using a full SVD after explicitly constructing weighted $\bm Z_i$ scales superquadratically, while our optimized version of the randomized SVD retains the same sub-quadratic scaling as computing the $\bm R_{BB,i}$. Importantly, the reduction error is still kept well below FEM error while the dimensionality of $\bm Q_i$ remains consistently within approximately $5\%$ of the direct SVD. Already for medium-sized problems our optimizations are a magnitude faster while requiring roughly half the memory of straightforward approaches. The performance gap increases with problem size in both cases, and will effect the requirements and costs of worker nodes in \cref{tbl:compenv}.

Next, \cref{fig:qtaqscaling} presents how our projection approach improves upon the direct projection by full matrix multiplications $\bm Q^T\bm A\bm Q$.
Both approaches rely on sparse matrix multiplication with the same number of nonzeroes, but the difference is that the explicit method constructs both $\bm Q$ and $\bm A$, while our optimized version constructs neither explicitly and multiplies only nonzero indices. Expectedly, the approaches have the same asymptotics, but our optimized version has a magnitude faster computational times. Crucially, the direct approach runs out of memory using the master node in \cref{tbl:compenv} already when degrees-of-freedom are in the single millions, while the optimized approach can project systems closer to 100 million degrees-of-freedom. Without optimizing the projection to the reduced basis, scaling the method is indeed impossible in memory-constrainted environments.

\begin{figure}
	\centering
	\begin{subfigure}[b]{0.75\textwidth}
		\centering
		\includegraphics[width=\textwidth]{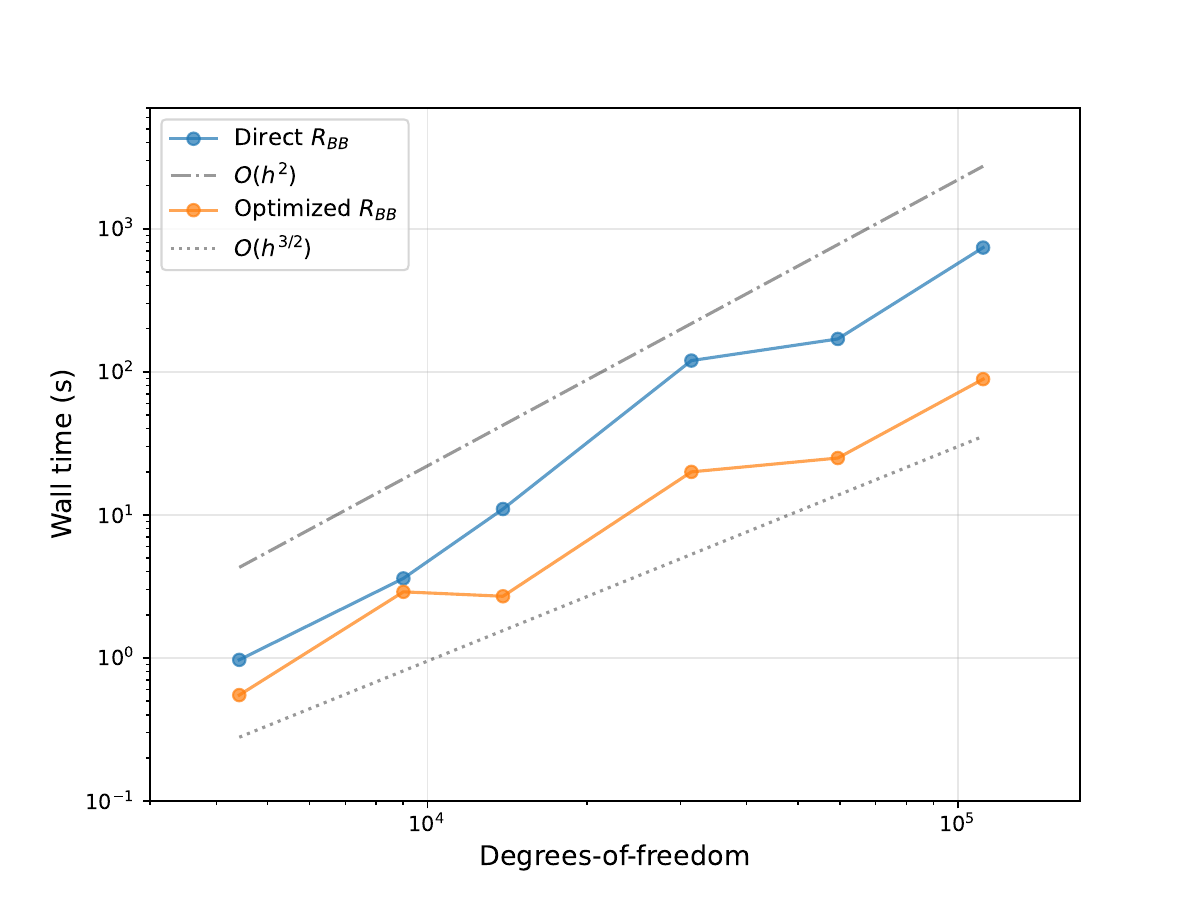}
		\caption[]%
		{{\small Wall clock time of computing $\bm R_{BB,i}$ in seconds.}}
	\end{subfigure}
	\hfill
	\begin{subfigure}[b]{0.75\textwidth}  
		\centering 
		\includegraphics[width=\textwidth]{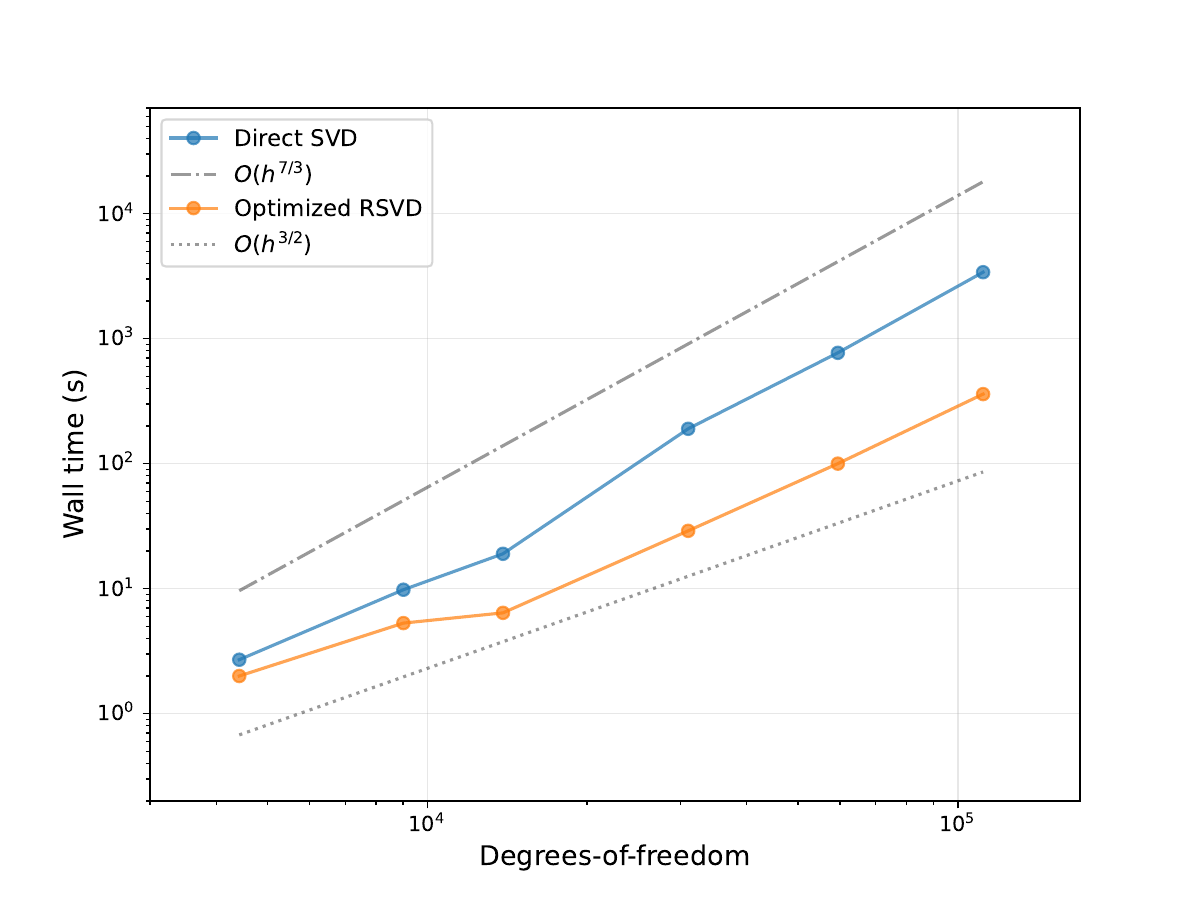}
		\caption[]%
		{{\small Wall clock time of computing $\bm Q_i$ in seconds.}}
	\end{subfigure}
	\caption[ Case 7 error ]
	{\small Scaling of local subproblems plotted on log-log scale with degrees-of-freedom on the $x$-axis and wall clock time in seconds on the $y$-axis. Our optimizations make computing both $\bm R_{BB,i}$ and $\bm Q_i$ sub-quadratic.}
	\label{fig:subproblemscaling}
\end{figure}

\begin{figure}
	\begin{center}
		\includegraphics[width=0.75\textwidth]{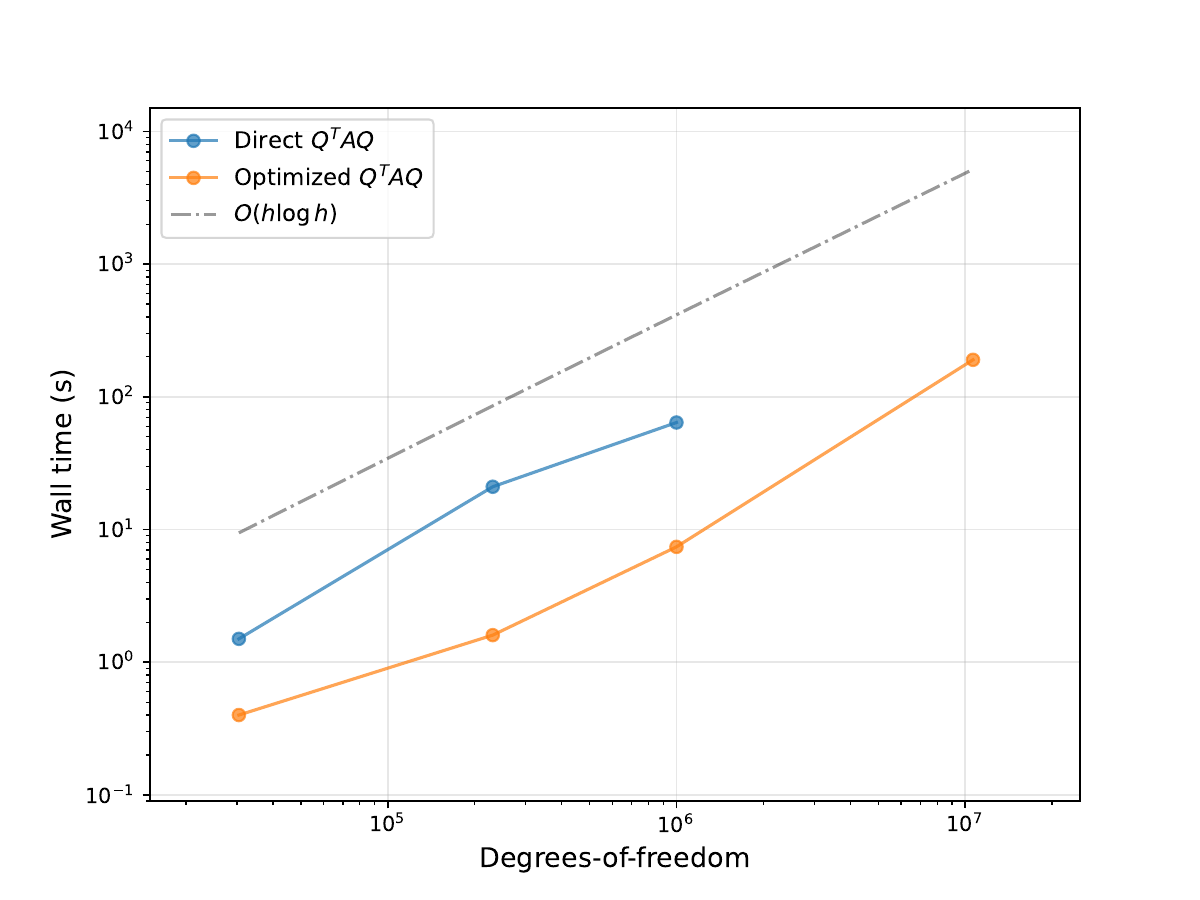}
	\end{center}
	\caption{Scaling of $\bm{Q^TAQ}$ in seconds of wall clock time on log-log scale. The straightforward transformation of the system to the reduced basis runs out of memory already in some millions of DOF, while our optimized approach is a magnitude faster and can project problems of almost 100 million DOF.}\label{fig:qtaqscaling}
\end{figure}

\subsection{Nonconstant $a$}
The error analysis in \cref{sec:analysis}
is done under the assumption that the material parameter in the governing equation \eqref{modelproblem} is $a=1$.
Thus, the analysis does not
answer to, e.g., how the different constants depend
on a spatially varying $a=a(x)$ and
whether the algorithm should be changed to accommodate this.
We will now present preliminary numerical results
using a space-dependent material field,
\[
a = a(x) = 10^k \sin(100 x) + 10^k + 1, \quad k=1,2,3,
\]
and leave more detailed error analysis for
future work. In the case of $k=3$, the ratio
between the largest and the smallest values of $a(x)$
is $2001 : 1$.

The finite element mesh of case 5 in \cref{tbl:scaling}
is suitable for resolving the oscillation
in $a(x)$.
Thus, the case 5 was computed repeatedly
for $k=1,2,3$ with the same $6$-hop extension and tolerance $\epsilon=\expnumber{1}{-2}$. The local bases were computed by explicitly constructing 
the weighted $\bm Z_i$ 
and their SVDs. The reduction errors were
$\expnumber{6.5}{-5}$, $\expnumber{2.0}{-5}$ and $\expnumber{6.3}{-6}$,
respectively. The $H^1$ norm of the solution is proportional to the reciprocal of $a$, 
which decreases the upper bound in \cref{thm:globalbound}.
This could explain the decrease in reduction error for higher $k$.
These preliminary results suggest that the proposed dimension reduction method is
not particularly sensitive to large variations in the material parameter.

\subsection{Engineering geometry}

%In the previous sections we have made
%the broad statement that the size of the
%subdomain interface will have an effect
%on the performance of the method.
Many engineering geometries are to some extent "2.5-dimensional"
in the sense that the interfaces between the subdomains
can be greatly reduced in size through the use
of fill-reducing graph partitionings.
Smaller interface sizes
are expected to further boost
the performance of the method, i.e.~lead
to smaller optimal reduced bases.
Thus, as a final example, we present results using a propeller geometry
from~\cite{ledoux_mambo-project_nodate}
and the loading
\[
f = f(x, y) = x^2 + y^2.
\]
Note that the axis $x=y=0$ corresponds to
the rotational axle of the propeller.

Using Gmsh~\cite{geuzaine_gmsh_2009}, we created a tetrahedral mesh with \numprint{850558}
nodes and \numprint{4394546} elements.
The geometry and the results of the basis
reduction are depicted in \cref{fig:propeller}.
The results show that for close to half of the 400 subdomains, the subdomains at the propeller blades which originally contain more than a thousand DOFs, are
reduced by two orders of magnitude to roughly ten or less DOFs
when using 6-hop extensions and the same tolerance
as in the scaling test, $\varepsilon=\expnumber{1}{-2}$, see the graph (c) in \cref{fig:propeller}.  Consequently, the system matrix
is reduced from $\numprint{850558} \times \numprint{850558}$ to $\numprint{10561} \times \numprint{10561}$ which corresponds to a reduction of $98.8\,\%$ of the DOFs. Note that the case 7
in the scaling tests results in a reduction
of $96.5\,\%$ from $\numprint{1000000}\times\numprint{1000000}$ to $\numprint{35041}\times\numprint{35041}$, so roughly three times less reduction.

%The solution time
%of the reduced linear system was negligible, only $0.22$ seconds.
%Depending on the subdomain, the
%basis reduction without the randomized variant took anywhere between
%$0.9$ and $9.9$ seconds per subdomain meaning
%that with 400 compute nodes the basis reduction 
%can be done in about 10 seconds.
%Taking into account the overheads from the construction of the
%reduced system and the data transfer,
%we estimate the whole 

\section{Conclusions and future work}

Model order reduction for finite element methods based on local approximation spaces have been the
focus of several papers in recent years, but its potential for distributed computing is yet to
be realized.
As a step forward, we have extended the numerical
analysis of the local approximation problems in terms of weighted $\ell_2$ norms and
derived straightforward error bounds. Consequently, the numerical error becomes simple to control 
using a single parameter, $\varepsilon > 0$, interpreted as the tolerance for the worst case reduction error in the $H^1$ norm. Furthermore, our work has clarified various
practicalities of implementing such schemes in a massively parallel computing environment.
As the computing nodes do not communicate during the model order reduction, the public cloud becomes 
a viable compute resource for the method.

The numerical experiments were implemented in the cloud, and included arbitrary mesh partitions and an approach for
transforming the global system efficiently using neighboring subdomain
information. The results displayed accuracy and conditioning properties
that can be expected from an implementation of
the piecewise linear finite element method. The reduction error concentrated on the
partition interfaces with near-perfect approximations inside the subdomains.
The high memory utilization of the graph partitioning algorithm seems to 
be the first serious bottleneck when scaling up the problem size.

Future work includes extending the analysis to more general partial differential equations and studying alternative approaches
for combining the reduced bases
over the subdomain interfaces.
The latter topic could both generalize the method to
higher order finite element methods and improve the partition-of-unity style error estimate to avoid the $h^{-1}$ scaling. 
The scheme could also be tested for even larger problems,
potentially utilizing accelerator hardware.

\begin{figure}[H]
	\centering
	\begin{subfigure}[b]{0.5\textwidth}  
		\centering 
		\includegraphics[width=\textwidth]{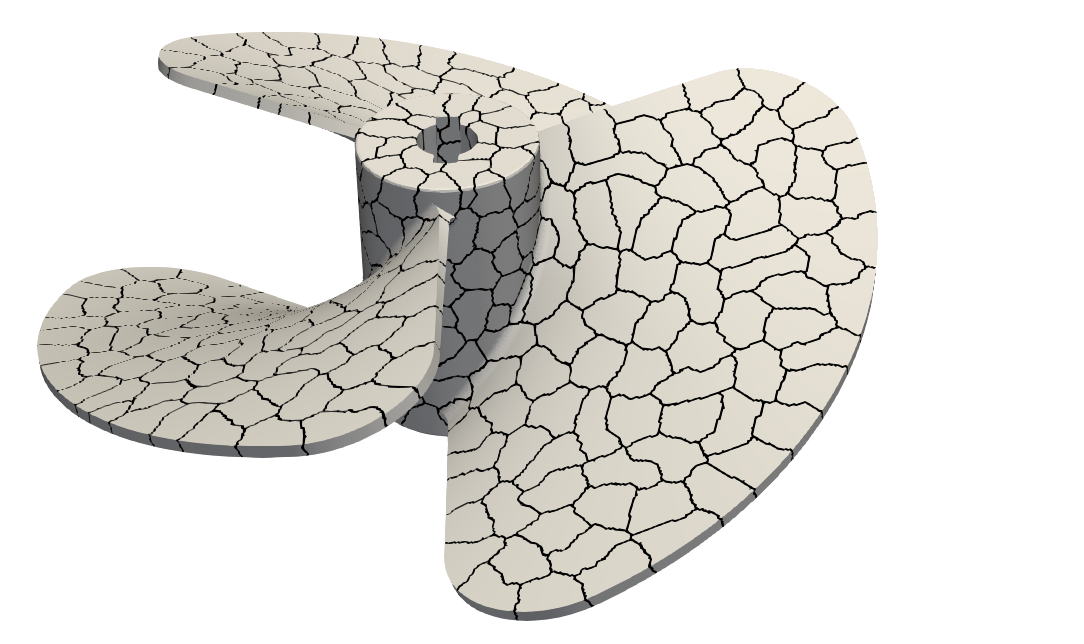}
		\caption[]%
		{{\small A visualization of $\omega_0$ and the subdomain boundaries.}}
	\end{subfigure}
	\hfill
	\begin{subfigure}[b]{0.5\textwidth}
		\centering
		\includegraphics[width=\textwidth]{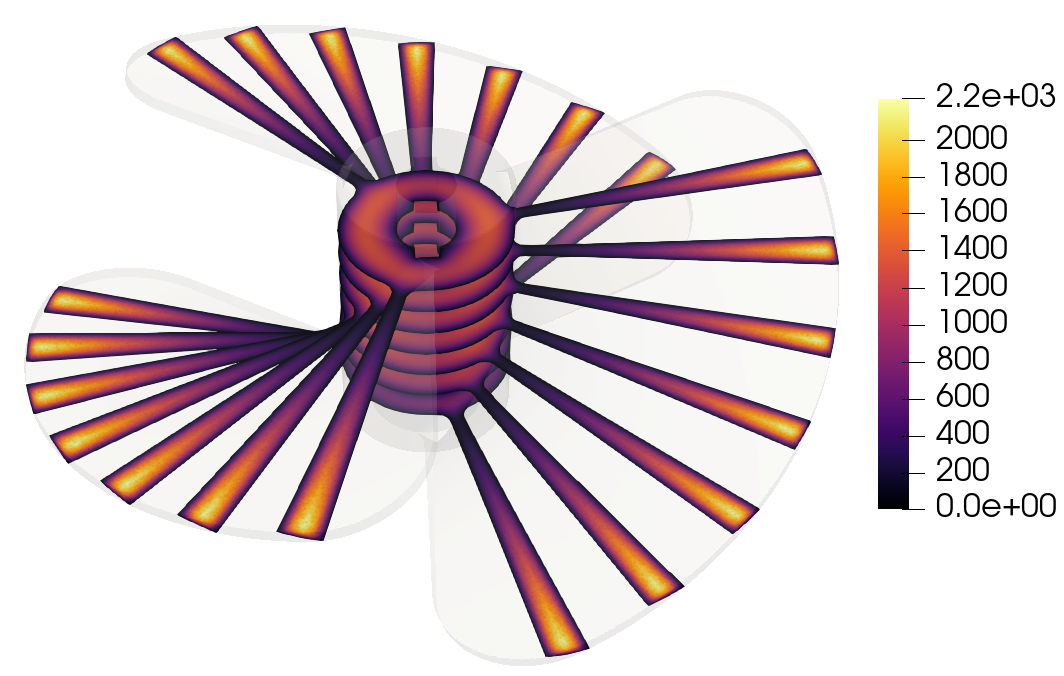}
		\caption[]%
		{{\small The reduced solution at a selection of propeller cross sections.}}
	\end{subfigure}
 	\begin{subfigure}[b]{0.8\textwidth}  
		\centering 
		\includegraphics[width=\textwidth]{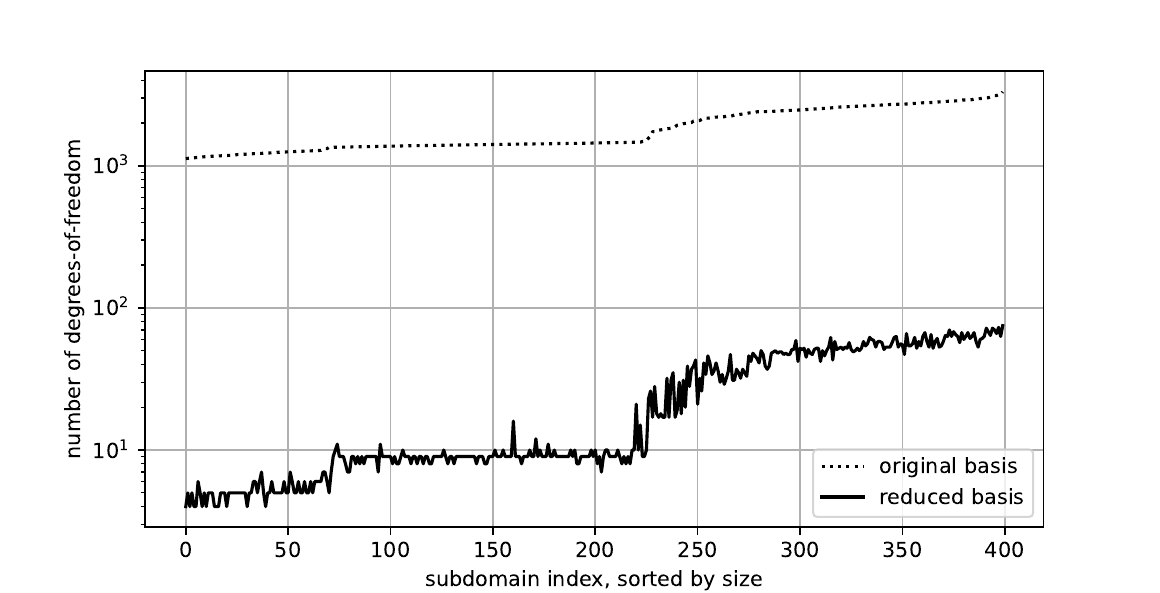}
		\caption[]%
		{{\small The magnitude of the basis reduction for each subdomain and $\varepsilon=\expnumber{1}{-2}$ on a logarithmic scale.}}
	\end{subfigure}
	\caption[ Case 7 error ]
	{\small The results of the engineering geometry example.  The blades of the propeller are "2.5-dimensional" with a small constant number of elements ($\sim 10$) in the thickness direction of the blades.}
	\label{fig:propeller}
\end{figure}

\clearpage

\appendix

\section{Proof to \cref{lem:boundarytol2norm}}

\label{proof:boundarytol2norm}

\begin{proof}
	The squared trace norm can be written as 
	\begin{align*}
		\norm[\partial V_i^+]{z}^2 & = \inf_{\substack{v \in V_{i}^+,\\ v|_{\partial \omega_{i}^+ \setminus \partial \Omega} = z}} \| v\|_{1,\omega_{i}^+}^2 \\
								   & = \inf_{\substack{v \in V_{i}^+,\\ v|_{\partial \omega_{i}^+ \setminus \partial \Omega} = z}} \norm[0,\omega_i^+]{v}^2 + \norm[0, \omega_i^+]{\nabla v}^2 \\
								   & = \min_{\substack{\bm\beta\in\R^{N_i+M_i},\\ \bm\beta_B = \bm\beta_z}} \bm\beta^T \bm B_i^+ \bm \beta + \bm\beta^T \bm A_i^+ \bm\beta \\
								   & = \inf_{\substack{\bm\beta\in\R^{N_i+M_i},\\ \bm\beta_B = \bm\beta_z}} \vec{\bm\beta_I, \bm\beta_z}^T \mat{\bm A_{II}, \bm A_{IB}; \bm A_{BI}, \bm A_{BB}}\vec{\bm\beta_I, \bm\beta_z},
	\end{align*}
where $\bm\beta$ is the coefficient vector of $v$ and its boundary indices are constrained to $\bm\beta_z$. Because $\bm B_i^+$ is symmetric PD and $\bm A_i^+$ symmetric PSD, $\bm A$ is symmetric PD. The optimization problem can then be solved uniquely by differentiating with respect to $\bm\beta_I$ and setting the derivative to zero:
	\begin{align*}
		\nabla_{\bm\beta_I} (\bm\beta^T \bm A\bm\beta) & = \nabla_{\bm\beta_I} (\bm\beta_I^T \bm A_{II}\bm\beta_I + 2\bm\beta_I^T \bm A_{IB}\bm\beta_z + \bm\beta_z^T \bm A_{BB}\bm\beta_z)\\
				& = 2\bm A_{II}\bm\beta_I + 2\bm A_{IB}\bm\beta_z \\
													   & = 0 \\
		\implies \bm\beta_I & = -\bm A_{II}^{-1}\bm A_{IB}\bm\beta_z.
	\end{align*}
	Substituting and using \cref{lem:blockchol} we get
	\begin{align*}
		\norm[\partial V_i^+]{z}^2 & = \min_{\bm\beta_I\in\R^{N_i}}\bm\beta_I^T \bm A_{II}\bm\beta_I + 2\bm\beta_I^T \bm A_{IB}\bm\beta_z + \bm\beta_z^T \bm A_{BB}\bm\beta_z \\
			& = \bm\beta_z^T \bm A_{BI}\bm A_{II}^{-1}\bm A_{IB}\bm\beta_z -2\bm\beta_z^T \bm A_{BI}\bm A_{II}^{-1}\bm A_{IB}\bm\beta_z + \bm\beta_z^T \bm A_{BB}\bm\beta_z \\
			& = \bm\beta_z^T (\bm A_{BB} - \bm A_{BI}\bm A_{II}^{-1}\bm A_{IB})\bm\beta_z \\
			& = \bm\beta_z^T \bm R_{BB}^T \bm R_{BB}\bm\beta_z \\
			& = \norm[\ell_2]{\bm R_{BB}\bm\beta_z}^2,
	\end{align*}
	and taking square roots finishes the proof.
\end{proof}

\section{Probabilistic error bound for the low-rank approximation}

\label{error:lowrank}

Per \cite[Theorem 10.6]{halko_finding_2011}, the expected error of the randomized low-rank approximation is 
\begin{lemma}
	\label{lem:experror}
	Let $\bm A\in R^{m\times n}$ with singular values $\sigma_1\geq \sigma_2\geq \cdots$. Choose a target rank $k\geq 2$ and an oversampling parameter $p\geq 2$, where $k+p\leq \min\Set{m,n}$. Draw a standard Gaussian matrix $\bm \Omega \in \R^{n \times (k+p)}$ and construct the sample matrix $\bm Y = \bm A\bm \Omega$. Then the expected approximation error is 
	$$\mathbb{E}\norm[\ell_2]{(\bm I - \bm P_{\bm Y})\bm A} \leq \left(1+\sqrt{\frac{k}{p-1}}\right)\sigma_{k+1} + \frac{e\sqrt{k+p}}{p}\left(\Sigma_{j>k} \sigma_j^2\right)^{1/2}.$$
\end{lemma}
We also include the tail error result \cite[Theorem 10.8]{halko_finding_2011}. 
\begin{lemma}
	\label{lem:tailerror}
	Consider the assumptions of \cref{lem:experror}. Assume further that $p\geq 4$. For all $u,t\geq 1$, $$\norm[\ell_2]{(\bm I - \bm P_{\bm Y})\bm A}\leq \left(1+t\cdot \sqrt{\frac{3k}{p+1}}\right)\sigma_{k+1} + t\cdot \frac{e\sqrt{k+p}}{p+1}(\Sigma_{j>k}\sigma_j^2)^{1/2} + ut\cdot \frac{e\sqrt{k+p}}{p+1}\sigma_{k+1},$$ with failure probability at most $2t^{-p} + e^{-u^2/2}$.
\end{lemma}

For our example at the end of \cref{sec:scalingtests}, the $6$-hop case of \cref{fig:spectrums} (A) with the sketching parameter $k_i = \lfloor M_i/8 \rfloor$, the minimum rank is $k\approx 40$ and oversampling parameter $p\approx 175-k = 135$. The expected error bound for sketching would then be approximately $$\frac{3}{2}\sigma_{k+1} + \frac{1}{4}(\Sigma_{j>k}\sigma_j^2)^{1/2}$$ and the tail error bound with $t=2, u=5$ approximately $$\frac{11}{2}\sigma_{k+1} + \frac{1}{2}(\Sigma_{j>k}\sigma_j^2)^{1/2}$$ with failure probability roughly $\expnumber{4}{-6}$. 

Again, we iterate that while these are the error bounds for the low-rank approximation, the local error bound derived from the low-rank approximation in \cref{lem:localbound} is very loose. Hence, the actual reduction errors present in \cref{sec:numtests} are magnitudes smaller.

\section{Spectra of weighted $\bm Z_i$}
\label{app:spectra}

\Cref{fig:spectrums} presents singular values of weighted lifting operators $\bm Z_i$
corresponding to two differently sized subdomains and different extensions 
plotted with a logarithmic $y$-axis. Recall that the lifting operator is
defined via \cref{prob:aux2} and is load-independent. The decay is first
exponential, then closer to polynomial, and then again exponential. Noticeably,
increased extensions escalate the decay rate, but the singular values go
ultimately to machine epsilon at the same number of singular vectors regardless
of the extension size. This number is very close to the degrees-of-freedom of
the original interface of the subdomain. Intuitively, all solution information
of a zero load problem in the subdomain is encoded in the interface
degrees-of-freedom, and the weighted $\bm Z_i$ closely represents this
phenomenon.
% We assume this behaviour emerges from the finite accuracy of the FEM; given a smoothening map and a finite image dimension, maps from some finite-dimensional domain span the image. 
The original subdomain size does not seem to really
affect the rate of decay given the same relative extension. Smaller extensions
would result in easier local problems at a cost of slightly more difficult
global problems as more boundary dimensions must be included in the
approximation. For extensions that roughly double the subdomain diameter, like
$6$-hop in \cref{fig:spectrums}, using $k_i = \lfloor M_i/8 \rfloor$ as the sketching
parameter seems suitable for yielding optimal bases with tolerance
$\epsilon=\expnumber{1}{-2}$ with a very high probability.

\begin{figure}[H]
	\centering
	\begin{subfigure}[b]{0.7\textwidth}
		\centering
		\includegraphics[width=\textwidth]{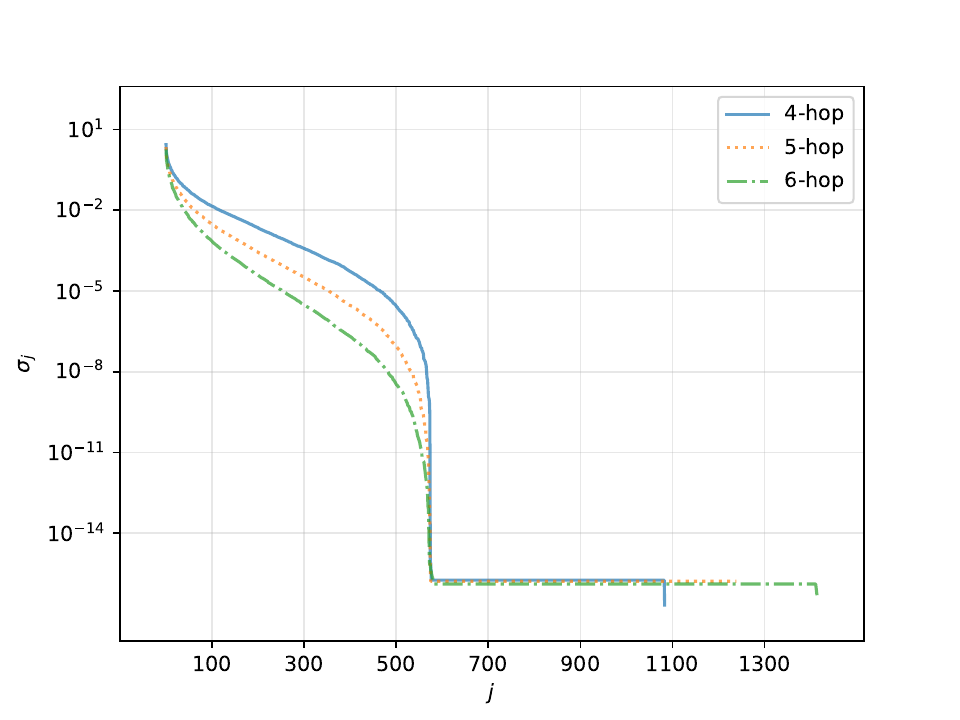}
		\caption[]%
		{{\small $\#\omega=\numprint{2519}$ with extensions containing \numprint{6170}--\numprint{9283} DOFs }}
	\end{subfigure}
	\hfill
	\begin{subfigure}[b]{0.7\textwidth}  
		\centering 
		\includegraphics[width=\textwidth]{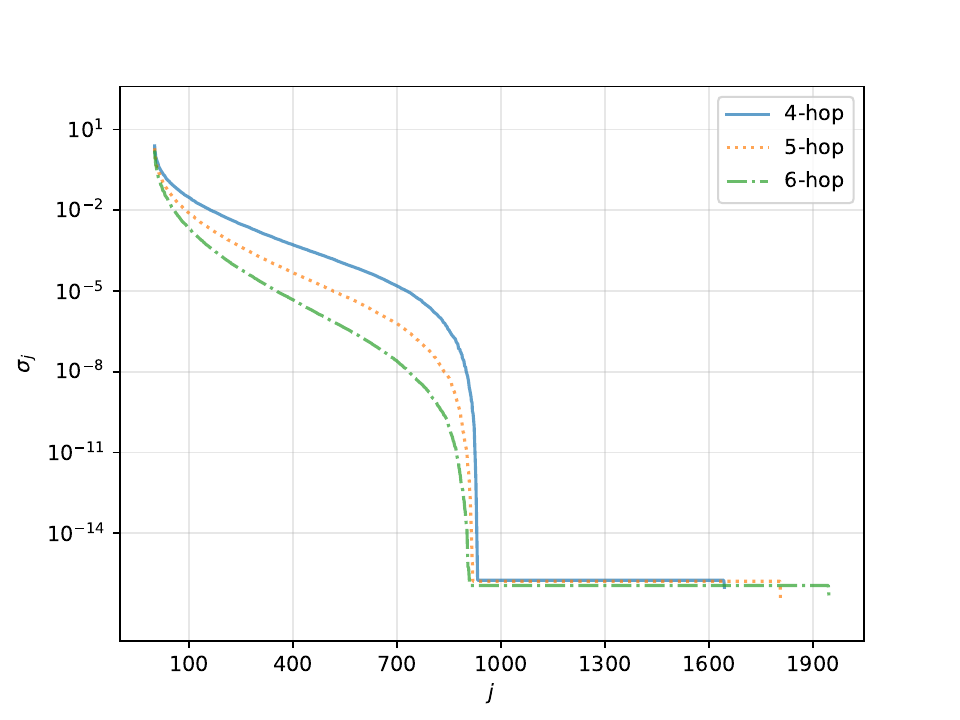}
		\caption[]%
		{{\small $\#\omega=\numprint{4895}$ with extensions containing \numprint{11875}--\numprint{17148} DOFs }}
	\end{subfigure}
	\caption[ Z operator singular values ]
	{\small The sorted singular values of weighted $\bm Z_i$ for two subdomains with $4$, $5$ and $6$-hop extensions plotted on a logarithmic $y$-axis.}
	\label{fig:spectrums}
\end{figure}

% Even the largest case in \cref{fig:spectrums} was solved in under a minute.

\bibliography{references}

\begin{thebibliography}{10}

\bibitem{babuska_optimal_2011}
Ivo Babuska and Robert Lipton.
\newblock Optimal {Local} {Approximation} {Spaces} for {Generalized} {Finite}
  {Element} {Methods} with {Application} to {Multiscale} {Problems}.
\newblock {\em Multiscale Modeling \& Simulation}, 9(1):373--406, January 2011.

\bibitem{babuska_partition_1997}
I.~Babuška and J.~M. Melenk.
\newblock The {Partition} of {Unity} {Method}.
\newblock {\em International Journal for Numerical Methods in Engineering},
  40(4):727--758, 1997.

\bibitem{babuska_machine_2014}
Ivo Babuška, Xu~Huang, and Robert Lipton.
\newblock Machine {Computation} {Using} the {Exponentially} {Convergent}
  {Multiscale} {Spectral} {Generalized} {Finite} {Element} {Method}.
\newblock {\em ESAIM: Mathematical Modelling and Numerical Analysis},
  48(2):493--515, March 2014.

\bibitem{babuska_multiscale-spectral_2020}
Ivo Babuška, Robert Lipton, Paul Sinz, and Michael Stuebner.
\newblock Multiscale-{Spectral} {GFEM} and optimal oversampling.
\newblock {\em Computer Methods in Applied Mechanics and Engineering},
  364:112960, June 2020.

\bibitem{brenner_mathematical_2008}
Susanne~C. Brenner and L.~Ridgway Scott.
\newblock {\em The {Mathematical} {Theory} of {Finite} {Element} {Methods}},
  volume~15 of {\em Texts in {Applied} {Mathematics}}.
\newblock Springer, New York, NY, 2008.

\bibitem{buhr_randomized_2018}
Andreas Buhr and Kathrin Smetana.
\newblock Randomized {Local} {Model} {Order} {Reduction}.
\newblock {\em SIAM Journal on Scientific Computing}, 40(4):A2120--A2151,
  January 2018.

\bibitem{calo_randomized_2016}
Victor~M. Calo, Yalchin Efendiev, Juan Galvis, and Guanglian Li.
\newblock Randomized {Oversampling} for {Generalized} {Multiscale} {Finite}
  {Element} {Methods}.
\newblock {\em Multiscale Modeling \& Simulation}, 14(1):482--501, January
  2016.

\bibitem{coghlan_magellan_2011}
Susan Coghlan.
\newblock The {Magellan} {Final} {Report} on {Cloud} {Computing}.
\newblock December 2011.

\bibitem{dolean_introduction_2015}
Victorita Dolean, Pierre Jolivet, and Frédéric Nataf.
\newblock {\em An {Introduction} to {Domain} {Decomposition} {Methods}:
  {Algorithms}, {Theory}, and {Parallel} {Implementation}}.
\newblock Society for Industrial and Applied Mathematics, Philadelphia, PA,
  November 2015.

\bibitem{efendiev_generalized_2013}
Yalchin Efendiev, Juan Galvis, and Thomas~Y. Hou.
\newblock Generalized multiscale finite element methods ({GMsFEM}).
\newblock {\em Journal of computational physics}, 251:116--135, 2013.

\bibitem{efendiev_generalized_2014}
Yalchin Efendiev, Juan Galvis, Guanglian Li, and Michael Presho.
\newblock Generalized multiscale finite element methods: Oversampling
  strategies.
\newblock {\em International Journal for Multiscale Computational Engineering},
  12(6), 2014.

\bibitem{geuzaine_gmsh_2009}
Christophe Geuzaine and Jean‐François Remacle.
\newblock Gmsh: {A} 3‐{D} finite element mesh generator with built‐in
  pre‐ and post‐processing facilities.
\newblock {\em International Journal for Numerical Methods in Engineering},
  79(11):1309--1331, September 2009.

\bibitem{gustafsson_scikit-fem_2020}
Tom Gustafsson and G.~D. McBain.
\newblock scikit-fem: {A} {Python} package for finite element assembly.
\newblock {\em Journal of Open Source Software}, 5(52):2369, August 2020.

\bibitem{halko_finding_2011}
N.~Halko, P.~G. Martinsson, and J.~A. Tropp.
\newblock Finding {Structure} with {Randomness}: {Probabilistic} {Algorithms}
  for {Constructing} {Approximate} {Matrix} {Decompositions}.
\newblock {\em SIAM Review}, 53(2):217--288, January 2011.

\bibitem{hannukainen_distributed_2022}
Antti Hannukainen, Jarmo Malinen, and Antti Ojalammi.
\newblock Distributed {Solution} of {Laplacian} {Eigenvalue} {Problems}.
\newblock {\em SIAM Journal on Numerical Analysis}, 60(1):76--103, February
  2022.

\bibitem{karypis_metis_1997}
George Karypis and Vipin Kumar.
\newblock {METIS}: {A} {Software} {Package} for {Partitioning} {Unstructured}
  {Graphs}, {Partitioning} {Meshes}, and {Computing} {Fill}-{Reducing}
  {Orderings} of {Sparse} {Matrices}.
\newblock Report, 1997.

\bibitem{sourcepackage}
Vili Kohonen and Tom Gustafsson.
\newblock {Numerical experiments for "Distributed finite element solution using
  model order reduction"}, April 2024.

\bibitem{ledoux_mambo-project_nodate}
Franck Ledoux.
\newblock {MAMBO}-project: {Model} database mesh blocking.

\bibitem{ma_error_2022}
Chupeng Ma and Robert Scheichl.
\newblock Error estimates for discrete generalized {FEMs} with locally optimal
  spectral approximations.
\newblock {\em Mathematics of Computation}, 91(338):2539--2569, 2022.

\bibitem{ma_novel_2022}
Chupeng Ma, Robert Scheichl, and Tim Dodwell.
\newblock Novel {Design} and {Analysis} of {Generalized} {Finite} {Element}
  {Methods} {Based} on {Locally} {Optimal} {Spectral} {Approximations}.
\newblock {\em SIAM Journal on Numerical Analysis}, 60(1):244--273, February
  2022.

\bibitem{martinsson_randomized_2020}
Per-Gunnar Martinsson and Joel~A. Tropp.
\newblock Randomized numerical linear algebra: {Foundations} and algorithms.
\newblock {\em Acta Numerica}, 29:403--572, 2020.

\bibitem{munhoz_faulttolerant_2022}
Vanderlei Munhoz, Márcio Castro, and Odorico Mendizabal.
\newblock Strategies for fault-tolerant tightly-coupled {HPC} workloads running
  on low-budget spot cloud infrastructures.
\newblock In {\em 2022 IEEE 34th International Symposium on Computer
  Architecture and High Performance Computing (SBAC-PAD)}, pages 263--272,
  2022.

\bibitem{pellegrini_scotch_1996}
François Pellegrini and Jean Roman.
\newblock Scotch: {A} software package for static mapping by dual recursive
  bipartitioning of process and architecture graphs.
\newblock In Heather Liddell, Adrian Colbrook, Bob Hertzberger, and Peter
  Sloot, editors, {\em High-{Performance} {Computing} and {Networking}},
  Lecture {Notes} in {Computer} {Science}, pages 493--498, Berlin, Heidelberg,
  1996. Springer.

\bibitem{roy_craig_coupling_nodate}
Jr. Roy~Craig.
\newblock Coupling of substructures for dynamic analyses - {An} overview.
\newblock In {\em 41st {Structures}, {Structural} {Dynamics}, and {Materials}
  {Conference} and {Exhibit}}. American Institute of Aeronautics and
  Astronautics.

\bibitem{schleus_optimal_2022}
Julia Schleuß and Kathrin Smetana.
\newblock Optimal {Local} {Approximation} {Spaces} for {Parabolic} {Problems}.
\newblock {\em Multiscale Modeling \& Simulation}, 20(1):551--582, March 2022.

\bibitem{toselli_domain_2005}
Andrea Toselli and Olof~B. Widlund.
\newblock {\em Domain {Decomposition} {Methods} — {Algorithms} and {Theory}},
  volume~34 of {\em Springer {Series} in {Computational} {Mathematics}}.
\newblock Springer, Berlin, Heidelberg, 2005.

\bibitem{verfurth_posteriori_2013}
Rüdiger Verfürth.
\newblock {\em A {Posteriori} {Error} {Estimation} {Techniques} for {Finite}
  {Element} {Methods}}.
\newblock Oxford University Press, April 2013.

\bibitem{wohlmuth_discretization_2001}
Barbara~I. Wohlmuth.
\newblock {\em Discretization {Methods} and {Iterative} {Solvers} {Based} on
  {Domain} {Decomposition}}, volume~17 of {\em Lecture {Notes} in
  {Computational} {Science} and {Engineering}}.
\newblock Springer Berlin Heidelberg, Berlin, Heidelberg, 2001.

\bibitem{zhou_hardware_2007}
Jiazheng Zhou, Xuan-Yi Lin, and Yeh-Ching Chung.
\newblock Hardware supported multicast in fat-tree-based {InfiniBand} networks.
\newblock {\em The Journal of Supercomputing}, 40(3):333--352, June 2007.

\end{thebibliography}
\bibliographystyle{plain}

\end{document}